\documentclass[a4paper,10pt]{article}
\usepackage[utf8x]{inputenc}
\usepackage[frenchb]{babel}
\usepackage[T1]{fontenc}
\usepackage{amsmath,amsfonts,amssymb,mathrsfs,latexsym,amscd,amsthm}
\usepackage{dsfont}
\usepackage{verbatim}
\usepackage[onehalfspacing]{setspace}
\theoremstyle{plain}

\DeclareMathOperator{\pgcd}{\operatorname{pgcd}}

\DeclareMathOperator{\p}{\mathbb{P}_{\mathbb{Q}}^2}
\DeclareMathOperator{\h}{H}
\DeclareMathOperator{\e}{\textbf{e}}

\newtheorem{theo}{Théorème}[section]
	\newtheorem{lemme}[theo]{Lemme}
	\newtheorem{propo}[theo]{Proposition}
	\newtheorem{cory}[theo]{Corollaire}
	\newtheorem{rmk}[theo]{Remarque}

	\newtheorem{defi}[theo]{Définition}

\newcommand\aess{\alpha_{\operatorname{ess}}}
\title{Distribution locale des points rationnels de hauteur bornée sur une surface de del Pezzo de degré 6}
\author{Zhizhong Huang\\
	\multicolumn{1}{p{.7\textwidth}}{\centering{\textit{Institut Fourier}\\ \textit{Université Grenoble Alpes}\\ \texttt{zhizhong.huang@univ-grenoble-alpes.fr}}}}
\date{}

\begin{document}
	\maketitle

\begin{abstract}
Nous établissons une mesure décrivant de manière précise la distribution locale asymptotique des points rationnels 
en dehors des sous-variétés localement accumulatrice
autour d'un point rationnel général sur une surface de del Pezzo de degré 6 au sens de l'approximation diophantienne.
\end{abstract}
 \section{Introduction}

 L'étude des points de hauteur bornée sur les variétés algébriques 
 reste un sujet majeur de la géométrie diophantienne. Le principe 
 de Manin et Batyrev prédit une formule asymptotique du nombre total de ces points
 et la distribution globale attendue de ces points est donné par le principe d'équidistribution formulé dans \cite{Peyre2},
 mais elle ne révèle pas les aspects d'approximation diophantienne mis en évidence dans les travaux de 
 McKinnon et Roth (\cite{Mc} et \cite{M-R}).
 Dans cet article, inspiré par les résultats connus, nous proposons un cadre
 pour étudier la distribution locale en un point général dans une variété algébrique (cela a été considéré en premier dans 
 un travail non publié de Pagelot \cite{Pagelot}). 
 Il ressort de notre étude l'existence, pour certaines surfaces toriques au moins, d'une mesure de distribution locale limite pour
 le facteur de \og zoom\fg\ critique.
 Cette mesure et l'ordre de grandeur du terme principal reflètent certaines propriétés géométriques de la variété.
 
  \subsection{Notations et définitions}
 \subsubsection{Énoncé du problème}
 
 Une surface $X_3$ de del Pezzo de degré 6 définie sur le corps des nombres rationnels est isomorphe sur $\mathbb{C}$ à 
 l'éclatement de  $\mathbb{P}^2$ en trois points en position générale.
 Dans le cas où ces points sont rationnels, nous pouvons nous ramener au cas où les trois points sont donnés en coordonnées homogènes par $P_1=[1:0:0]$, $P_2=[0:1:0]$, $P_3=[0:0:1]$.
On vérifie alors que $X_3$ est isomorphe à une sous-variété de $\mathbb{P}^1\times\mathbb{P}^1\times\mathbb{P}^1$ définie par l'équation $u_1 u_2 u_3=v_1 v_2 v_3$, où $[u_i,v_i]$ est la coordonnée homogène du $i$-ème $\mathbb{P}^1$.
On note $\omega_X^{-1}$ le fibré en droites anticanonique de $X_3$. Avec ce plongement $\omega_X^{-1}$ est isomorphe au tiré en arrière du fibré $\mathcal{O}(1,1,1)$.
On note $\pi:X_3\rightarrow \mathbb{P}^2$ le morphisme d'éclatement et $E_1,E_2,E_3$ les diviseurs exceptionnels.
Puisque $\pi$ est un isomorphisme en dehors $E_1\cup E_2\cup E_3$, 
on abrège souvent $\pi^{-1}([\alpha:\beta:\gamma])$ en $[\alpha:\beta:\gamma]$ pour $\alpha\beta\gamma\neq 0$.

Pour tout point $P=\pi^{-1}([x:y:z])$ qui n'appartient pas aux diviseurs exceptionnels, on utilise la hauteur de $P$ calculée par la formule suivante:
$$H_{\omega_X^{-1}}(\pi^{-1}([x:y:z]))=\frac{\max(|x|,|y|)\max(|y|,|z|)\max(|x|,|z|)}{\operatorname{pgcd}(x,y)\operatorname{pgcd}(y,z)\operatorname{pgcd}(x,z)}.$$
L'application $H=H_{\omega_X^{-1}}$ est une hauteur de Weil absolue associée à $\omega_X^{-1}$.
Comme les autres surfaces de del Pezzo de degré $\geqslant 7$, la formule de la hauteur 
(associée au fibré en droites anticanonique ample) reflète le nombre de points éclatés.

Le groupe $G=\mathbb{G}_m^3/\mathbb{G}_m$ agit sur $X_3$ de la manière suivante
$$(\lambda,\mu,\delta)\times [x:y:z]\longmapsto [\lambda x: \mu y: \delta z]$$ 
Pour $\alpha \beta \gamma \neq 0$, il existe un élément de $G$ qui envoie $[\alpha:\beta:\gamma]$ sur $Q=[1:1:1]$.
On peut donc se ramener pour notre étude au point neutre $Q$.

Il y a trois courbes $\{l_i\}_{i=1}^3$ passant par $Q$, définies par les transformations strictes des trois droites dans $\mathbb{P}^2$
d'équations respectives:
$$ x=y,\quad y=z,\quad x=z.$$
On les appelle sous-variétés localement accumulatrices en $Q$.
Nous verrons qu'elles contiennent des points rationnels \og très proches\fg\  dans un voisinage de $Q$.
On note $Z=\cup_i l_i$.
Sur $X_3\diagdown Z$, les points rationnels s'approchent plus lentement de $Q$ que ceux dans $Z$ (en un sens que nous allons préciser),
mais ils ne se concentrent pas sur les sous-variétés strictes. 
On s'intéresse à la manière dont ils sont distribués.

Une façon d'attaquer ce problème est de définir une suite de mesures et de trouver son comportement asymptotique (comme indiqué dans \cite{Pagelot}).
Après avoir fixé un voisinage $U$ assez petit de $Q$ et fixé un difféomorphisme local $\rho$ de la surface en $Q$ sur le plan tangent 
$T_Q X_3$ qui envoie $Q$ sur l'origine,
on considère la famille de mesures $\{\varrho_{B,Q}\}_B$ définies par
\begin{equation}
 \varrho_{Q,B}=\sum_{x\in U(\mathbb{Q}),H(x)\leqslant B}\delta_{B^{\frac{1}{\aess(Q)}}\rho(x)}
\end{equation}
où $\delta_P$ désigne la mesure de Dirac en point $P$
et $\aess(Q)$ est la constante d'approximation essentielle de $Q$ (cf. Définition \ref{de:app}).
Intuitivement, en multipliant par une puissance de $B$, on fait un \og zoom\fg\  sur l'image de $U(\mathbb{Q})$ dans l'espace tangent de $Q$ en fonction de la hauteur.
Dans plusieurs d'exemples la quantité totale $N(U,B)=\varrho_{B,Q}(\chi_U)$ a un terme principal de la forme $ B^a (\log B)^b$ où 
$\chi(U)$ désigne la fonction caractéristique du voisinage compact $U$ de $Q$ dans $T_Q X_3$. 
Ce terme signifie une \og grandeur\fg\ des points rationnels autour de $Q$ atteignant la constante d'approximation essentielle.
On va aussi étudier la convergence de la suite au sens usuel, c'est-à-dire la convergence de la suite réelle $\{N(U,B)^{-1}\int f \varrho_{B,Q}\} $ 
pour une fonction $f$ continue à support compact. Pour calculer la limite il suffit de 
diviser l'intégrale par le terme principal au lieu de $N(U,B)$ lui-même.
Si cette suite de mesures converge, la mesure limite $\varrho_{Q}$ donne asymptotiquement \og la densité des points rationnels tendant vers $Q$\fg. On appelle $\varrho$ \textit{mesure asymptotique de la distribution locale} ou en bref \textit{mesure asymptotique} au point $Q$. Plus précisément, 
$$\int f\operatorname{d}\varrho_{B,Q}\sim B^a(\log B)^b\int f d\varrho_{Q}.$$
On retrouvera l'analogue (sur les variétés) du principe fondamental de l'approximation diophantienne:
nous mesurons la qualité de l'approximation en fonction de la hauteur.

La puissance de \og zoom \fg\ est en fait \textit{a posteriori} une conséquence du calcul de la constante d'approximation essentielle $\aess(Q)$. 
On va voir que les droites passant par $Q$ ont une constante d'approximation exactement égale à $\aess(Q)$.
C'est un fait empirique que les courbes rationnelles dans une variété projective uni-réglée contiennent les points rationnels 
plus proche d'un point rationnel fixé (comme conjecturé par McKinnon dans \cite{Mc} (Conjecture 2.7)). 

On fixe quelques notations pour une utilisation ultérieure. La lettre $p$ désigne un nombre premier.
 Rappelons qu'une fonction $f:\mathbb{N}\rightarrow \mathbb{C}$ est dite \textit{multiplicative} si elle vérifie les conditions suivantes:
 $$f(1)=1,~ f(mn)=f(m)f(n) ~ \mbox{si}  ~ \pgcd(m,n)=1.$$
 On note $\mu$ la fonction de Möbius, $\tau$ la fonction donnant le nombre total de diviseurs, $\omega$ la fonction donnant le nombre total de facteurs premiers, $\varphi$ la fonction d'indicatrice d'Euler et on définit la fonction $\phi$ par
 \begin{equation}
  \phi(n)=\prod_{p|n}\left(1-\frac{1}{p}\right).
 \end{equation}
 On a $\varphi(n)=n\phi(n)$.
\subsubsection{La distance projective}
On introduit une fonction de distance pour $X_3$. Elle est déjà définie pour toute place d'un corps de nombre dans (\cite{M-R} \S2).
Pour $x=[x_0:\cdots,x_n],y=[y_0:\cdots:y_n]\in\mathbb{P}^n(\mathbb{R})$, la distance projective naturelle (associée à la place infinie) 
$d_\infty(x,y)$ est définie par
\begin{equation}\label{eq:dist}
 d_\infty(x,y)=\frac{\sum_{0\leqslant i\neq j\leqslant n}|x_i y_j-x_j y_i|}{\sqrt{\sum_{i=0}^n |x_i|^2}\sqrt{\sum_{i=0}^n |y_i|^2}}.
\end{equation}
Pour une sous-variété projective $Y$ de $\mathbb{R}^n$, \textit{la distance projective naturelle} sur $Y$ est obtenue par restriction.
Les distances induites par des immersions différentes sont équivalentes dans un voisinage compact de $Q$ (\cite{M-R} Lemma 2.4).
On appelle une fonction $d:Y(\mathbb{R})\times Y(\mathbb{R})\mapsto \mathbb{R}_{\geqslant 0}$ \textit{une distance projective} si elle équivaut à une distance projective naturelle sur $Y$.

Dans notre cas on définit la fonction $d:X_3(\mathbb{Q})\times X_3(\mathbb{Q})\to \mathbb{R}$ de la façon suivante.
On plonge $X_3$ dans $\mathbb{P}^1\times \mathbb{P}^1\times \mathbb{P}^1$.
Pour $W=[s_1:t_1]\times[s_2:t_2]\times[s_3:t_3]$ et $V=[x_1:y_1]\times[x_2:y_2]\times[x_3:y_3]$, on définit
\begin{align*}
 d(W,V)&=d_{12}(W,V)+d_{23}(W,V)+d_{13}(W,V)\\
&=\frac{|s_1 y_1-t_1 x_1|}{\sqrt{s_1^2+t_1^2}\sqrt{x_1^2+y_1^2}}+
\frac{|s_2 y_2-t_2 x_2|}{\sqrt{s_2^2+t_2^2}\sqrt{x_2^2+y_2^2}}+
\frac{|s_3 y_1-t_3 x_1|}{\sqrt{s_3^2+t_3^2}\sqrt{x_3^2+y_3^2}}.
\end{align*}
Par un calcul de routine on voit que $d$ équivaut à la distance projective naturelle sur $X_3$ induite par l'immersion de Segre.
Donc $d$ est une distance projective. La raison pour laquelle on choisit cette distance est que l'on peut déduire une borne inférieure pour le produit de $d$ avec certaine puissance de la hauteur anticanonique (Voir la Section 2).
\subsubsection{Les constantes d'approximation}
Plusieurs constantes d'approximation sont définies dans (\cite{M-R} Définition 2.7~2.10 et \cite{Pagelot}). 
Nous les rappelons ici.

Soient $X$ une variété projective lisse et $V$ une sous-variété (ouverte ou fermée) de $X$, $L$ un fibré en droites ample et $Q\in V(\mathbb{\bar{Q}})$. Soit $H_L$ une hauteur de Weil associée à $L$.
On considère les ensembles 

\begin{equation}
 A(Q,V)=\{\gamma>0: \exists C>0 ,\exists (y_i)\subset V(\mathbb{Q}), d(Q,y_i)\to 0, d(Q,y_i)^\gamma H_L(y_i)<C\},
 \end{equation}
 \begin{equation}
 B(Q,V)=\{\gamma>0: \exists C>0,d(Q,y)\geqslant C H_{L}(y)^{-\frac{1}{\gamma}},y\in  V(\mathbb{Q}),y\neq Q\}.
\end{equation}

Il est évident que si $\gamma_0 \in B(Q,V)$, tout $0<\gamma<\gamma_0$ appartient à $B(Q,V)$. 
De même, si $\gamma_0 \in A(Q,V)$, tout $\gamma>\gamma_0$ appartient à $A(Q,V)$. 
\begin{propo}
 \begin{equation}
  \inf A(Q,V)=\sup B(Q,V).
 \end{equation}
\end{propo}
La démonstration est essentiellement celle de (\cite{M-R}, Proposition 2.11) en remarquant que $L$ est ample et donc vérifie la propriété de Northcott.

\begin{defi}\label{de:app}
 On définit \emph{constante d'approximation} en $Q$ dans $V$ comme $$\alpha(Q,V)=\inf A(Q,V).$$
Alors $A(Q,V)$ est un intervalle contenant $]\alpha(x,V),+\infty[$.
On définit \emph{constante d'approximation essentielle}, en abrégé \emph{constante essentielle}, en $Q$ par la quantité
\begin{equation}
 \aess(Q)=\sup_V \alpha(Q,V)
\end{equation}
où $V$ parcourt toutes les parties ouvertes de $X$ telles que $d(Q,V)=\inf\{d(Q,y):y\in V(\mathbb{Q})\}=0$. 
S'il existe une sous-variété $Z$ de $X$ contenant $Q$ de sorte que pour tout ouvert $W$ de $Z$, $\alpha(Q,W)<\aess(Q)$, on dit que $Z$ est \textit{localement accumulatrice}. Si de plus $\alpha(Q,X\setminus Z)=\aess (Q)$, on dit que $Z$ est la variété localement accumulatrice \textit{maximale}.
\end{defi}
Les deux constantes ne dépendent ni du choix de la fonction de distance, ni du choix de la hauteur associée à $L$.
Intuitivement, plus la constante d'approximation est petite, meilleure est l'approximation.
\begin{theo}[\cite{M-R}, Lemme 2.14, Proposition 2.15]\label{th:3}
	Soit $x\in\mathbb{P}^n(\mathbb{Q})$. Pour une hauteur de Weil associée à $\mathcal{O}(d)$, on a 
	$$\alpha(x,\mathbb{P}^n)=\aess(x)=d.$$ 
\end{theo}

Maintenant revenons à notre cas, à savoir $Q$ désigne le point $[1:1:1]$ dans $X=X_3$ et $L=\omega_X^{-1}$ est le fibré anticanonique.
Pour chaque $l_i$, on a 
$$\omega_X^{-1}|_{l_i}=\mathcal{O}_{l_i} ((\omega_X^{-1}\cdot l_i))=\mathcal{O}_{l_i} (2).$$
Par la fonctorialité de la hauteur (\cite{Serre} 2.3), $H_{\omega_X^{-1}}|_{l_i}$ équivaut à la hauteur associée au fibré en droites $\mathcal{O}(2)$ de $l_i\simeq\mathbb{P}^1$.
Le Théorème \ref{th:3} implique que $\alpha(Q,Z)=2$ (rappelons qu'ici $Z$ est la réunion des trois droites $l_i)$.

Notre premier résultat, démontré dans le \S2, est le suivant.
\begin{theo}\label{th:1}
On a $3\in A(Q,X_3\diagdown Z)$ et $\alpha(Q,X_3\diagdown Z)=3$.
En particulier, $\aess(Q)=3$ et $Z$ est la variété localement accumulatrice maximale.
\end{theo}
Ainsi, pour trouver une partie ouverte de $X_3$ dont la constante d'approximation atteint la constante essentielle,
it faut supprimer toutes les sous-variétés localement accumulatrice, à savoir les trois droites. 
\subsection{Énoncé du théorème principal}
Comme $Z$ est localement accumulatrice, on la supprime et on étudie la distribution locale des points rationnels sur $X_3\diagdown Z$.
Nous pouvons toujours modifier le voisinage $U$ de $Q$ de sorte que tous les points $[x:y:z]$ de $U(\mathbb{Q})$ vérifient
$$x>0,\quad y>0, \quad z>0.$$
Les trois droites définies au début divisent $U$ en six parties (selon les d'ordre induit de $\{x,y,z\}$). 
Les symétries naturelles nous permettent de nous ramener à l'ensemble $R\cap U$ où 
$$R=\{[x:y:z]:x>y>z>0\}.$$ On fixe des coordonnées $(s,t)$ de $T_Q X_3$,
le difféomorphisme local $\rho: U(\mathbb{Q})\to T_Q X_3$ que l'on va utiliser est donné sous ces coordonnées par
\begin{equation}
 [x:y:z]\mapsto \left(\frac{x}{z}-1,\frac{y}{z}-1\right)\in \mathbb{R}^2.
\end{equation}
Ainsi $Q$ est envoyé sur $(0,0)$ et l'image de $R$ dans $\mathbb{R}^2$ par le difféomorphisme $\rho$ ci-dessus
est la région 
$$\bar{R}=\{(s,t)\in \mathbb{R}^2:s>t>0\}.$$
On introduit la distance euclidienne sur $T_Q X_3$ définie par 
\begin{equation}
 d((s,t),(w,z))=\max\left(|s-w|,|t-z|\right).
\end{equation}
Cette distance est équivalente à la distance projective définie précédemment.
\begin{theo}\label{th:2}
 Pour toute fonction $f$ intégrable à support compact dans la région $$\bar{R}=\{(s,t)\in \mathbb{R}^2: s>t>0\},$$ on a pour $B\to\infty$,
 \begin{equation}
  \int f\operatorname{d}\varrho_{B,Q}=B^{\frac{1}{3}}\log B\int f(s,t)\frac{\beta(st(s-t))}{st(s-t)}\operatorname{d}s \operatorname{d} t + O(B^{\frac{1}{3}}\log\log B)
 \end{equation}
où $$\beta(s)=\frac{2}{\pi^2}\sum_{\substack{u,e\\u^3e<s}}u\phi(ue)3^{\omega(e)}\sum_{k|u}\frac{\mu(k)3^{\omega(k)}}{ k}$$
et $\phi,\omega,\mu$ sont des fonctions arithmétiques définies au début. La constante implicite dépend de $f$. 
\end{theo}

La mesure asymptotique a une densité, alors que dans les cas de $\mathbb{P}^2,X_1$ et $X_2$ 
(ici $X_i$ signifie la surface obtenue par éclater $\mathbb{P}^2$ en $i$ points en position générale), 
la mesure asymptotique est concentrée sur les droites, dont la dimension de Hausdorff est $1$
 (\cite{Pagelot}).
De plus, la fonction de densité fait apparaître les trois droites qui sont les variétés localement accumulatrices qu'on a retirées.
On va discuter de ce phénomène à la fin du texte.
\subsection{Les méthodes utilisées pour le comptage}
La preuve du Théorème \ref{th:1} est élémentaire grâce au choix de la distance.
Cette preuve est faite au \S2.
Les détails du comptage forment le \S4.
La difficulté significative en comparant ce type de problème aux problèmes classiques de comptage des points rationnels sur les variétés algébriques est l'apparence de la distance. 
Puisque chaque point rationnel est situé sur une droite unique passant par $Q$ 
(comme précédemment, cela veut dire la transformation stricte des droites dans $\mathbb{P}^2$), 
on introduit une paramétrisation naturelle pour ces droites. 
Cela peut, dans une certaine mesure, simplifier l'expression des pgcd. 
Ensuite on va définir quelques paramètres concernant les pgcd et décomposer l'ensemble considéré suivant ces paramètres grâce à l'observation 
que dans un voisinage borné de $Q$, certains de ces paramètres ne prennent qu'un nombre fini de valeurs qui est indépendant de la borne de la hauteur.
Puis en fixant ces paramètres, on compte sur chaque droite et on somme pour obtenir la quantité totale.
Au cours de cette étape la transformation de Cremona joue un rôle important.
Elle envoie tous les points ayant une \og grande\fg\ pente sur ceux dans l'autre région, dont la pente est
\og petite \fg, si bien que l'on peut les traiter directement.
Enfin pour obtenir une mesure, on va échanger des sommes et des intégrales.

Notre méthode peut être adaptée facilement aux cas $\mathbb{P}^2,X_1$, et $X_2$.
\section{La preuve du Théorème \ref{th:1}}
\subsection{Borne inférieure}
On établit d'abord une borne inférieure pour $\alpha(Q,X)$ et $\aess(Q)$.

On fixe un point $P=\pi^{-1}([x:y:z])\in X_3\smallsetminus Z$. Par symétrie on peut supposer que $x>y>z$.
On introduit les notations:
$$u_1=\frac{x}{\pgcd(x,y)},~~~~u_2=\frac{y}{\pgcd(x,y)},$$
$$v_1=\frac{y}{\pgcd(y,z)},~~~~v_2=\frac{z}{\pgcd(y,z)},$$
$$w_1=\frac{x}{\pgcd(x,z)},~~~~w_2=\frac{z}{\pgcd(x,z)}.$$

Alors on a $u_1>u_2,~v_1>v_2,~w_1>w_2$. Avec ces notations, $H(P)=u_1 v_1 w_1$. En utilisant l'inégalité fondamentale $$\frac{1}{n}(\sum_{i=1}^n a_i)\geqslant \sqrt[n]{\prod_{i=1}^n a_i},$$ on obtient
\begin{align*}
 & d(P,Q)H_{\omega_X^{-1}}(P)^{\frac{1}{3}}\\
 & =\frac{1}{\sqrt{2}}\left(\frac{|u_1-u_2 |}{\sqrt{u_1^2+u_2^2}}+
\frac{|v_1 -v_2|}{\sqrt{v_1^2+v_2^2}}+
\frac{|w_1-w_2|}{\sqrt{w_1^2+w_2^2}}\right) (u_1 v_1 w_1)^{\frac{1}{3}}\\
&\geqslant \frac{1}{2}\left((u_1-u_2)u_1^{-\frac{2}{3}}v_1^{\frac{1}{3}}w_1^{\frac{1}{3}}
+(v_1-v_2)u_1^{\frac{1}{3}}v_1^{-\frac{2}{3}}w_1^{\frac{1}{3}}
+(w_1-w_2)u_1^{\frac{1}{3}}v_1^{\frac{1}{3}}w_1^{-\frac{2}{3}}\right)\\
&\geqslant \frac{3}{2}\sqrt[3]{(u_1-u_2)(v_1-v_2)(w_1-w_2)}\\
&\geqslant \frac{3}{2}
\end{align*}

Cela montre que $3\in A(Q,X_3\diagdown Z)$ et donc on a $\aess(Q)\geqslant \alpha(Q,X_3\diagdown Z)\geqslant 3$.

\subsection{Borne supérieure}

Prenons une droite $D$ différente de $l_i$, on a 
$$\omega_X^{-1}|_D=\mathcal{O}_D ((\omega_X^{-1}\cdot D))=\mathcal{O}_D (3).$$
Et $H_{\omega_X^{-1}}|_D$ équivaut à la hauteur associée au fibré en droites $\mathcal{O}(3)$ de $D\simeq\mathbb{P}^1$.
D'après le Théorème \ref{th:3}, $\alpha(Q,D)=3$. Puisque $\cup_{D\neq l_i} D$ est dense dans $X_3$, pour toute partie ouverte $V$ de $X_3$ de sorte que $d(Q,V)=0$ on a
$\alpha(Q,V)\leqslant 3$, donc $\aess(Q)\leqslant 3$.
\subsection{Une remarque sur un résultat de McKinnon et Roth}

Supposons pour le moment que $X$ est une variété projective définie sur corps algébriquement clos et $L$ est un fibré en droites ample sur $X$.
On note $Q$ un point rationnel de $X$.

Rappelons que \textit{la constante de Seshadri} (pour $L$ et $Q$) est définie par 
$$\varepsilon(Q)=\inf_{Q\in C\subseteq X} \frac{(L,C)}{\operatorname{mult}_Q C}$$
où l'infimum est pris sur toutes les courbes intègres passant par $Q$.

\begin{defi}
 On dit que $\alpha(Q,X)$ \textit{est calculée} sur une sous-variété fermée stricte $Z$ de $X$ si $\exists \varepsilon_0>0$ tel que 
pour toute suite de points rationnels $(P_i)\to Q$ de $X$ vérifiant $\alpha(Q,(P_i))\leqslant \varepsilon_0+\alpha(Q,X)$,
tous sauf un nombre fini d'entre eux se trouvent dans $Z$. Autrement dit, $\alpha(Q,Z)<\alpha(Q,X\setminus Z)$. Dans ce cas là, 
$Z$ est localement accumulatrice.
\end{defi}
Il convient de distinguer les deux notions \og est calculée\fg\ et \og peut être calculée\fg. La première signifie que toute suite de points ayant la constante d'approximation assez petite se trouve uniquement dans $Z$ (ou bien, tout sauf un nombre fini d'éléments), alors que la deuxième dit simplement que l'on peut choisir une suite qui atteint la constante d'approximation donnée dans une sous-variété. Le corollaire suivant donne un exemple sur la différence entre ces notions.

Rappelons un des principaux théorèmes dans (\cite{M-R}). 
Il nous donne une condition suffisante pour que $\alpha$ soit calculée sur une sous-variété fermée stricte de $X$.
En d'autres termes, il fournit un critère de l'existence d'une sous-variété localement accumulatrice.  
\begin{theo}[\cite{M-R}, Théorème 6.2]\label{thm2}
On note $n=\dim X$. Alors $\alpha(Q,X)$ est calculée sur une sous-variété fermée stricte de $X$
pourvu que $\alpha(Q,X)<\frac{n}{n+1}\varepsilon_Q$.
\end{theo}

Maintenant, revenons à notre exemple précédent à savoir $X=X_3$ une surface de del Pezzo de degré 6 et $L$ le fibré anticanonique de $X_3$. 

\begin{cory}
La constante d'approximation $\alpha(Q,X_3)$, qui est égale à $2$, est calculée sur la sous-variété des trois droites $Z$. La constante essentielle $\aess(Q)$ peut être calculée sur toute droite passant par $Q$ différente de celle dans $Z$.
\end{cory}
\begin{proof}[Démonstration]
 Cela découle de $\alpha(Q,X_3\diagdown Z)=3,  \alpha(Q,Z)=2$ et de la discussion dans la Section 2.2.
\end{proof}

\begin{rmk}
 D'après (\cite{Broustet}), $\varepsilon(Q)=2$. Donc $\alpha(Q,X_3)>\frac{2}{3}\varepsilon(Q)$.
Cela signifie que l'hypothèse dans Théorème \ref{thm2} ci-dessus n'est pas vérifiée ici, alors qu'il 
existe des sous-variétés localement accumulatrices.
\end{rmk}

 \section{Outils: la transformation de Cremona et la paramétrisation des droites}
\subsection{La transformation de Cremona}
 On définit une application rationnelle $\psi:\mathbb{P}^2\dashrightarrow\mathbb{P}^2$ par
\begin{equation}\label{eq:Cremona}
 \psi:[x:y:z]=[yz:xz:xy].
\end{equation}
Cette application $\psi$ est bien définie sauf en les trois points $P_1,P_2,P_3$.  Elle est d'ordre $2$ et induit un morphisme:
$$\Psi: X_3\to X_3$$
$$\Psi([u_1:v_1]\times[u_2:v_2]\times[u_3,v_3])=[v_1:u_1]\times[v_2:u_2]\times[v_3,u_3].$$
En dehors des trois diviseurs exceptionnels $E_i$, $\Psi$ est calculé par la formule (\ref{eq:Cremona}).
Ce morphisme établit une bijection entre les points rationnels dans la région $R$ et ceux dans la région 
\begin{equation}\label{regions}
S=\{[x:y:z]:z>y>x>0\}.
\end{equation}
Par un calcul très simple que l'on omet, on a 
\begin{propo}\label{po:Cremona1}
La transformation de Cremona préserve la hauteur:
 $$\h\circ \Psi=\h.$$
\end{propo}
\subsection{La paramétrisation}
 Tout d'abord on introduit quelques paramètres intrinsèques pour les points rationnels dans $R$.
\begin{defi}
 Pour $V=[x:y:z]\in  R$ avec $\pgcd(x,y,z)=1$, on note
 \begin{equation}\label{eq:pgcd}
     d_1=\pgcd(y,z),\quad d_2=\pgcd(x,z),\quad d_3=\pgcd(x,y),\quad d_4=\pgcd(x-z,y-z),
 \end{equation}

  $$ m(V)=\frac{y-z}{d_4}, \quad n(V)=\frac{x-z}{d_4},$$
  \begin{equation}\label{eq:e}
    e_1=\frac{y-z}{d_1 d_4},\quad e_2=\frac{x-z}{d_2 d_4},\quad e_3=\frac{x-y}{d_3 d_4}.
  \end{equation}
on voit que 
$$\pgcd(e_i,e_j)=1.$$
La \textit{pente} est définie par 
\begin{equation}\label{eq:pente}
 \mu(V)=\frac{y-z}{x-z}=\frac{m(V)}{n(V)}.
\end{equation}
On va noter $u(V)=d_4$ et $\e(V)$ le triplet $(e_1,e_2,e_3)$ pour le point $V$. 
\end{defi}
Les nombres correspondants pour les points rationnels dans $S$ sont définies en échangeant $x$ et $z$. 
Pour éviter toute confusion, on les donne en détails.
Pour $U=[x:y:z]\in  S$ avec $\pgcd(x,y,z)=1$, on définit
  $$ d_1^\prime=\pgcd(y,x),\quad d_2^\prime=\pgcd(x,z),\quad d_3^\prime=\pgcd(z,y),\quad d_4^\prime=\pgcd(x-z,y-z),$$
  $$m^\prime(V)=\frac{y-x}{d_4^\prime}, \quad n^\prime(V)=\frac{z-x}{d_4^\prime},$$
  $$e_1^\prime=\frac{y-x}{d_1^\prime d_4^\prime},\quad e_2^\prime=\frac{z-x}{d_2^\prime d_4^\prime},\quad e_3^\prime=\frac{z-y}{d_3^\prime d_4^\prime}.$$ 
 $$\mu^\prime(V)=\frac{y-x}{z-x}=\frac{m^\prime(V)}{n^\prime(V)}.$$
 $$u^\prime(U)=d_4^\prime, \quad \e^\prime(V)=(e_1^\prime,e_2^\prime,e_3^\prime).$$
 On décrit maintenant comment la transformation de Cremona se traduit sur ces paramètres.
  \begin{propo}\label{po:Cremona2}
  Soit $V=[x:y:z]\in  R$. On note $\e(V)=(e_1,e_2,e_3)$ et $U=\Psi(V)\in S$. Alors on a
 $$u(V)= u^\prime(U), \quad \e^\prime(U)=(e_3,e_2,e_1),$$
 $$m^\prime(U)=\frac{e_1e_2e_3z}{m(V)n(V)},\quad n^\prime(U)=\frac{e_1e_2e_3y}{m(V)(n(V)-m(V))}.$$
  \end{propo}
\begin{proof}[Démonstration]
 On a $$U=[\frac{yz}{d_1d_2d_3}:\frac{xz}{d_1d_2d_3}:\frac{xy}{d_1d_2d_3}].$$ 
 Ces coordonnées de $U$ sont premières entre elles.
 On vérifie directement que 
 \begin{align*}
  u^\prime(U)&=\pgcd\left(\frac{y(x-z)}{d_1d_2d_3},\frac{x(y-z)}{d_1d_2d_3}\right)\\
  &=\pgcd(x-z,y-z)=u(V).
 \end{align*}
 $$d_1^\prime=\frac{z}{d_1d_2},\quad d_2^\prime=\frac{y}{d_1d_3},\quad d_3^\prime=\frac{x}{d_2d_3},\quad d_4^\prime=d_4.$$
 $$e_1^\prime=\frac{1}{d_1d_2d_3}\frac{xz-yz}{d_1^\prime d_4^\prime}=\frac{x-y}{d_3d_4}=e_3,$$
  $$e_2^\prime=\frac{1}{d_1d_2d_3}\frac{xy-yz}{d_2^\prime d_4^\prime}=\frac{x-z}{d_2d_4}=e_2,$$
 $$e_3^\prime=\frac{1}{d_1d_2d_3}\frac{xy-xz}{d_3^\prime d_4^\prime}=\frac{y-z}{d_1d_4}=e_1.$$
 $$m^\prime(U)=\frac{z(x-y)}{d_1d_2d_3d_4}=\frac{e_1e_2e_3z}{m(V)n(V)},$$
 $$n^\prime(U)=\frac{y(x-z)}{d_1d_2d_3d_4}=\frac{e_1e_2e_3y}{m(V)(n(V)-m(V))}.$$
\end{proof}

 Maintenant on introduit une paramétrisation pour les droites passant par $Q$ (sauf $l_i(1\leqslant i\leqslant 3)$).
 En termes de cette paramétrisation l'expression des pgcd est relativement simple.
 Toutes ces droites sont les transformations strictes des droites dans $\p $ de la forme 
 $$ax+by+cz=0 , \quad a+b+c=0, \quad abc\neq0.$$
 On peut toujours supposer que $\pgcd(a,b)=1$.
 Pour une telle droite $l$, on définit le morphisme de paramétrisation sur la région $R$ par:
  \begin{align*}
   \phi_{a,b}:~& \mathbb{P} ^1 \rightarrow l\subset X_3\\
              & [u:v]\mapsto [ub+v:-ua+v:v].
  \end{align*}
  Pour la région $S$, it suffit d'échanger les coordonnées.
   \begin{align*}
   \phi^\prime_{a,b}:~& \mathbb{P} ^1 \rightarrow l\subset X_3\\
              & [u:v]\mapsto [v:-ua+v:ub+v].
              \end{align*}
              
  \begin{propo}\label{po:relation}
   Il existe une bijection entre l'ensemble
   $$\{(a,b,u,v)\in\mathbb{Z}^4:0<-a<b,u>0,v>0,\pgcd(a,b)=\pgcd(u,v)=1\}$$
  et l'ensemble des points rationnels dans $R$ (resp. $S$).
  \end{propo}
\begin{proof}[Démonstration]
 Il suffit de montrer cela pour $R$ car les deux régions et les paramétrisations sont complètement symétriques.
 Une direction est déjà donnée par les paramétrisations.
 Pour l'inverse, on se donne un point rationnel $[x:y:z]\in R$ avec $\pgcd(x,y,z)=1$,
 on définit $0<-a<b$ et $u,v>0$ de sorte que
 $$-\frac{a}{b}=\frac{y-z}{x-z},\quad \frac{u}{v}=\frac{\frac{x}{z}-1}{b}=\frac{\frac{y}{z}-1}{-a}$$
 avec $$\pgcd(a,b)=\pgcd(u,v)=1.$$
 Cela veut dire que ce point correspond à un quadruplet unique $(a,b,c,d)$ vérifiant la condition souhaitée.
 Cela établit la bijection.
\end{proof}
Pour $V=[x:y:z]\in R\cap l$, on vérifie la relation suivante entre notre paramétrisation et les paramètres introduits précédemment.
  \begin{equation}\label{eq:relationabu}
   u=u(V),\quad a=-m(V),\quad b=n(V),
  \end{equation}
  $$d_1=\pgcd(-ua+v,v)=\pgcd(a,v),\quad e_1=-\frac{a}{d_1},$$
 $$d_2=\pgcd(ub+v,v)=\pgcd(b,v),\quad e_2=\frac{b}{d_2},$$
 $$d_3=\pgcd(ub+v,-ua+v)=\pgcd(ub+v,a+b)=\pgcd(-ua+v,a+b), \quad e_3=\frac{a+b}{d_3}.$$

Avant de faire le calcul, on remarque qu'une autre paramétrisation, qui est celle d'origine, est de paramétrer les coordonnées directement:
$$(z_{i,j},1\leqslant i<j\leqslant 5)\longmapsto (z_{1,4}z_{1,3}z_{2,5},z_{1,5}z_{1,3}z_{2,4},z_{1,4}z_{1,5}z_{2,3},z_{1,5},z_{1,4},z_{1,3})
=(x,y,z,d_1,d_2,d_3).$$
Elle vient du torseur universel au-dessus de $X_3$.
En l'utilisant on peut aussi éliminer les pgcd dans la formule de la hauteur.
Mais la difficulté viendra du calcul de la distance.
Rappelons que dans la formule de la distance sur le plan de tangent, l'expression ``$x-z$'' apparaît.
Cela va produire une condition linéaire sur les paramètres.
C'est cette condition qui ajoute certaines difficultés pour contrôler les termes d'erreur.
 \section{Le calcul global}
Pour obtenir une meilleure compréhension sur la distribution des points rationnels, on calcule d'abord la limite du
nombre total des points rationnels dans un voisinage compact de $Q$ dans $T_Q X_3$.
Comme indiqué précédemment, on peut se ramener à la région $R$.
On va calculer la limite de la suite $\{\varrho_{B,Q}(\chi(\varepsilon))\}_B$ où $\chi(\varepsilon)$ désigne la fonction caractéristique
du domaine $\{(x,y)\in \mathbb{R}^2:x>y>0, d(Q,(x,y))\leqslant \varepsilon\}\subset T_Q X_3$ avec $\varepsilon>0$ quelconque.
On va voir que pour $\varepsilon$ fixé,  on obtient la finitude de plusieurs paramètres $u(\cdot),e(\cdot)$, ce qui nous permet d'écrire la formule de somme d'une manière plus facile.

 On a \begin{equation}
       \varrho_{B,Q}(\chi(\varepsilon))=\sharp\left\{
 \begin{aligned}
  &P=[x:y:z]\\
  &x>y>z>0\\
  &\pgcd(x,y,z)=1
 \end{aligned}
\left|
\begin{aligned}
 &\max\left(\left|\frac{x}{z}-1\right|,\left|\frac{y}{z}-1\right|\right)\leqslant \varepsilon B^{-\frac{1}{3}}\\
 &\frac{\max(|x|,|y|)\max(|x|,|z|)\max(|y|,|z|)}{\pgcd(x,y)\pgcd(x,z)\pgcd(x,z)}\leqslant B
\end{aligned}
\right\}\right.
      \end{equation}
 On va sommer d'abord sur chaque droite. 
  \begin{equation}\label{eq:sommetotale}
   \varrho_{B,Q}(\chi(\varepsilon))=\sum_{\substack{0<-a<b\\\pgcd(a,b)=1}} \sharp F_B(a,b)
  \end{equation}
  où en introduisant les paramètres $u,v,a,b,d_1,d_2,d_3$ comme (\ref{eq:relationabu}) (pour $(a,b)$ fixé),
   \begin{align}
 F_B(a,b) &=\left\{
 \begin{aligned}
  &P=[x:y:z]\\
  &x>y>z>0\\
  &\pgcd(x,y,z)=1
 \end{aligned}
\left|
\begin{aligned}
 &\left|\frac{x}{z}-1\right|\leqslant \varepsilon B^{-\frac{1}{3}}; \mu(P)=-\frac{a}{b}\\
 &\frac{\max(|x|,|y|)\max(|x|,|z|)\max(|y|,|z|)}{\pgcd(x,y)\pgcd(x,z)\pgcd(x,z)}\leqslant B
\end{aligned}
\right\}\right.\\
  &=\left\{  
 \begin{aligned}
 & (u,v)\in \mathbb{N}^2\diagdown\{(0,0)\},\\
 &\pgcd(u,v)=1
 \end{aligned}
 \left| 
 \begin{aligned}
  &\frac{u}{v}\leqslant\varepsilon B^{-\frac{1}{3}}b^{-1} ~\textcircled{1}\
  &(ub+v)^2(-ua+v)\leqslant Bd_1d_2d_3~\textcircled{2}
 \end{aligned}
\right\} \right.
 \end{align}
 On note $e=e_1e_2e_3$ et $\lambda=-\frac{a}{b}(1+\frac{a}{b})$.
 D'une part, pour tout $(u,v)\in F_B(a,b)$, la condition \textcircled{2} nous donne 
 $$v^3<(ub+v)^2(-ua+v)\leqslant Bd_1d_2d_3.$$ Compte tenu de \textcircled{1}, 
$$u\leqslant \frac{\varepsilon v}{B^{\frac{1}{3}}b} < \frac{\varepsilon }{b}(d_1d_2d_3)^{\frac{1}{3}}
  = \frac{\varepsilon }{b}(|a|b(a+b))^{\frac{1}{3}}e^{-\frac{1}{3}}=\varepsilon (\lambda e^{-1})^{\frac{1}{3}}$$
ce qui entraîne que
$$\lambda \geqslant \frac{u^3e}{\varepsilon^3},$$
d'où
\begin{equation}\label{eq:slope}
 \left|\frac{a}{b}\right|\in \mathopen ]\frac{1}{2}-C_0,\frac{1}{2}+C_0\mathclose [.
\end{equation}
avec
\begin{equation}\label{eq:c0}
 C_0=C_0(u,e,\varepsilon)=\frac{1}{2}\sqrt{1-\frac{4eu^3}{\varepsilon^3}}.
\end{equation}
Cela nous donne une borne pour la pente des droites intervenant.

D'autre part, 
on a 
  \begin{equation}   
u\leqslant \varepsilon (\lambda e^{-1})^{\frac{1}{3}}<\frac{\varepsilon }{4^{\frac{1}{3}}}e^{-\frac{1}{3}}
  \end{equation}
 d'où \begin{equation}
       ue^{\frac{1}{3}} \leqslant \frac{\varepsilon }{4^{\frac{1}{3}}}.
      \end{equation}
Cela signifie que si l'on fixe $\varepsilon>0 $,
il n'y a qu'un nombre fini de choix pour $u$ et $e$ et donc pour $(e_1,e_2,e_3)$ aussi.
En particulier si $\varepsilon\leqslant 4^{\frac{1}{3}}$, $F(a,b)=\varnothing$.
Donc il y a un \og trou\fg \ autour de $Q$.
En remarquant que cette valeur est indépendante de $B$, cela correspond exactement à la \og borne inférieure\fg \ que l'on a démontré 
dans la section précédente (notons que l'on a changé la fonction de distance).

Nous pouvons écrire (\ref{eq:sommetotale}) sous la forme
\begin{equation}\label{eq:sommeF}
 \varrho_{B,Q}(\chi(\varepsilon))=\sum_{u,e_i} \sharp F(u,e_i,\varepsilon) =\sum_{u,e_i}\sum_{\substack{0<-a<b\\\pgcd(a,b)=1\\e_1|a,e_2|b,e_3|a+b}}\sharp F_{u,e_i,\varepsilon}(a,b)
\end{equation}
où 
   \begin{align}\label{eq:coutingwithfixedparameters}
   F(u,e_i,\varepsilon) &=\left\{
   \begin{aligned}
   &P=[x:y:z]\\
   &x>y>z>0\\
   &\pgcd(x,y,z)=1
   \end{aligned}
   \left|
   \begin{aligned}
   &\left|\frac{x}{z}-1\right|\leqslant \varepsilon B^{-\frac{1}{3}}, u(P)=u,\mathbf{e}(P)=(e_1,e_2,e_3);\\
   &\frac{\max(|x|,|y|)\max(|x|,|z|)\max(|y|,|z|)}{\pgcd(x,y)\pgcd(x,z)\pgcd(x,z)}\leqslant B
   \end{aligned}
   \right\}\right.\end{align} et
$F_{u,e_i,\varepsilon}(a,b)$ est l'ensemble des $ v\in \mathbb{N}\diagdown\{0\}$ vérifiant les deux conditions suivantes
\begin{equation}\label{eq:conditionaa}
 (ub+v)^2(-ua+v)\leqslant B|a|b(a+b)e^{-1},v\geqslant\varepsilon^{-1} B^{\frac{1}{3}}ub ~~\textcircled{a},
\end{equation}
\begin{equation}\label{eq:conditionbb}
 \pgcd(a,v)=-\frac{a}{e_1},\pgcd(b,v)=\frac{b}{e_2},\pgcd(-ua+v,a+b)=\frac{a+b}{e_3},\pgcd(u,v)=1~\textcircled{b}.
\end{equation}

Maintenant on fixe $\varepsilon,u,e_1,e_2,e_3$ avec $\pgcd(e_i,e_j)=1$ pour $1\leqslant i<j\leqslant3$.
Rappelons que $e=e_1e_2e_3$. Nous analysons les conditions \textcircled{a} et \textcircled{b} séparément.

\textbf{I. La condition \textcircled{a}}.\\
Puisque la condition 
$$(ub+v)^3\leqslant B|a|b(a+b)e^{-1}$$
implique la première condition de \textcircled{a}, qui elle-même entraîne que
$$v^3\leqslant B|a|b(a+b)e^{-1},$$
l'ensemble des solutions $v$ de la condition $\textcircled{a}$ est l'intersection avec $\mathbb{Z}$ de l'intervalle 
\begin{equation}\label{eq:interval}
I=I_{a,b,u,e,\varepsilon}=[\varepsilon^{-1} B^{\frac{1}{3}}ub,B^{\frac{1}{3}} e^{-\frac{1}{3}}(|a|b(a+b))^{\frac{1}{3}}-ubC_{a,b,u,B}]
\end{equation}
où $C_{a,b,u,B}\in\mathopen]0,1\mathclose[$ est une constante qui dépend de $a,b,B,u$. 
Pour que l'intervalle $I$ soit non-vide, on doit avoir 
$$e^{-\frac{1}{3}}B^{\frac{1}{3}} (|a|b(a+b))^{\frac{1}{3}}-ubC_{a,b,u,B}\geqslant\varepsilon^{-1} B^{\frac{1}{3}}ub.$$
Il en découle que 
\begin{equation}\label{eq:slope1}
   \left|\frac{a}{b}\right|\in \mathopen ]\frac{1}{2}-C_1,\frac{1}{2}+C_1\mathclose [.
 \end{equation}
 où 
\begin{equation}\label{eq:C1}
 C_1=C(a,b,u,\varepsilon,B)=\frac{1}{2}\sqrt{1-\frac{4eu^3}{\varepsilon^3}(1+\varepsilon C_{a,b,B}B^{-\frac{1}{3}})^3}
 <C_0.
\end{equation}

 avec
 \begin{equation}\label{eq:C1small}
  C_0-C_1=O_{\varepsilon,u,e}(B^{-\frac{1}{3}}).
 \end{equation}



\textbf{II. La condition \textcircled{b}}.
 La condition \textcircled{b} (\ref{eq:conditionbb}) implique la relation
\begin{equation}\label{eq:pgcdab}
\pgcd\left(u,\frac{ab(a+b)}{e}\right)=1.~\textcircled{c}
\end{equation} 
Nous restreignons donc la somme aux $(a,b)$ vérifiant cette condition.

Étant donné une paire $(p,q)$ telle que $$pa+qb=1,$$ on a
\begin{lemme}
Les $v$ satisfaisant la condition $\textcircled{b}$ (\ref{eq:conditionbb}) sont 
\begin{equation}\label{eq:latticecondition}
 \Gamma^\prime(a,b)=\{u(q-p)ab+ab(a+b)e^{-1}n:n\in\mathbb{Z},\pgcd(n,ue)=1\}.
\end{equation}
\end{lemme}
\begin{proof}[Démonstration]
 Les $v$ satisfaisant à
 $$\frac{a}{e_1}\left|v\right.,\quad\frac{b}{e_2}\left|v\right.,\quad\frac{a+b}{e_3}\left|-ua+v \right.$$
 constituent en le translaté d'un réseau 
 $$\Gamma(a,b)=\{u(q-p)ab+ab(a+b)e^{-1}n:n\in\mathbb{Z}\}.$$
 Puis les conditions sur les pgcd impliquent que $n$ doit être premier à $u$ et $e$.
 Réciproquement, puisque l'on a la condition (\ref{eq:pgcdab}), on vérifie sans difficulté que tous les éléments de $\Gamma^\prime(a,b)$ vérifient la condition \textcircled{b}.
\end{proof}
 On va démontrer une version préliminaire de la distribution globale avec quelque paramètres fixés.
  
\begin{propo}\label{po:global}
 On fixe $\varepsilon$, $u$, $e_1,e_2,e_3$ et on note $e=e_1e_2e_3$. Rappelons la définition de $F(u,e_i,\varepsilon)$ (\ref{eq:coutingwithfixedparameters}). On a
 \begin{equation}
  \sharp F(u,e_i,\varepsilon)=\sum_{\substack{0<-a<b\\ \pgcd(a,b)=1\\e_1|a,e_2|b,e_3|a+b}}\sharp F_{u,e_i,\varepsilon}(a,b)=
  \frac{2Z(\varepsilon,u,e)}{\pi^2}B^{\frac{1}{3}}\log B + O_{u,e,\varepsilon}(B^{\frac{1}{3}}\log \log B).
 \end{equation}
où
\begin{equation}\label{eq:Zeue}
 Z(\varepsilon,u,e)=\phi(ue)\sum_{k|u}\frac{\mu(k)3^{\omega(k)}}{ k}\left(\frac{1}{e^{\frac{1}{3}}}\int_{\theta\in \mathopen ]\frac{1}{2}-C_0,\frac{1}{2}+C_0\mathclose [} \frac{\operatorname{d}\theta}{(\theta(1-\theta))^{\frac{2}{3}}}-\frac{u}{\varepsilon}\int_{\theta\in \mathopen ]\frac{1}{2}-C_0,\frac{1}{2}+C_0\mathclose [} \frac{\operatorname{d}\theta}{(\theta(1-\theta))}\right).
\end{equation}
\end{propo}

Pour traiter une somme sur un réseau, it faut en déterminer l'indice. On fixe $(\alpha_1,\alpha_2,\alpha_3)\in (\mathbb{N}\diagdown \{0\})^3$
avec $\pgcd(\alpha_i,\alpha_j)=1$, on considère l'ensemble
\begin{equation}
H(\alpha_i)=\{(a,b,c)\in\mathbb{Z}^3:a+b+c=0,\alpha_1|a,\alpha_2|b,\alpha_3|c\}.
\end{equation}
C'est un sous-groupe de 
$$G=\{(a,b,c)\in\mathbb{Z}^3:a+b+c=0\}\simeq \mathbb{Z}^2.$$
La projection des deux premières variables
$$(a,b,c)\mapsto(a,b)$$
établit un isomorphisme des réseaux $H(\alpha_i)$ et
\begin{equation}\label{eq:isomorphismoflattices}
\{(a,b)\in \mathbb{Z}^2:\alpha_1|a,\alpha_2|b,\alpha_3|a+b\}.
\end{equation}
\begin{propo}
 \begin{equation}
  [G:H(\alpha_i)]=\alpha=\alpha_1\alpha_2\alpha_3.
 \end{equation}
\end{propo}
\begin{proof}[Démonstration]
Puisque $H=H(\alpha_i)$ est un sous-module de $G$, c'est un $\mathbb{Z}$-module libre. En tensorisant $H(\alpha_i)$ avec $\mathbb{Q}$,
on voit que le rang de $H(\alpha_i)$ est 2.
Par le théorème des restes chinois, le morphisme naturel 
$$G/H\rightarrow \oplus_{i=1}^3 G/H\otimes_\mathbb{Z} \mathbb{Z}/\alpha_i \mathbb{Z}$$
est un isomorphisme. 
De plus,
$$G/H\otimes_\mathbb{Z} \mathbb{Z}/\alpha_i \mathbb{Z}\simeq \mathbb{Z}/\alpha_i \mathbb{Z}.$$
Donc $$[G:H]=\sharp (\oplus_{i=1}^3  \mathbb{Z}/\alpha_i \mathbb{Z})=\alpha_1\alpha_2\alpha_3.$$
\end{proof}

D'après l'analyse précédente, la condition $\textcircled{a}$
nous fournit un intervalle (\ref{eq:interval}) et la condition $\textcircled{b}$ nous donne un réseau. 
Quand la longeur de l'intervalle est plus petite que la période du réseau, 
le cardinal que l'on cherche est $0$ ou $1$ et n'est pas équivalent à la longeur de l'intervalle.
C'est pourquoi ensuite nous allons discuter les deux cas séparément.

Pour $\varepsilon>u(4e)^\frac{1}{3}$, on définit une constante 
$$D=D(\varepsilon,u,e)=\sqrt{((4e)^{-\frac{1}{3}}-u\varepsilon^{-1})e(4^{-1}-C_0^2)^{-1}}>0.$$
On vérifie qu'on a pour $B\gg O_{u,e,\varepsilon}(1)$ et $b>DB^{\frac{1}{6}}$,
\begin{align*}
	|a|b(a+b)e^{-1}&=b^3\left(\left|\frac{a}{b}\right|\left(1-\left|\frac{a}{b}\right|\right)\right)e^{-1}\\
	&>b^3\left(\frac{1}{4}-C_0^2\right)e^{-1}\\
	&>bD^2B^{\frac{1}{3}}\left(\frac{1}{4}-C_0^2\right)e^{-1}\\
	&>bB^{\frac{1}{3}}\left(\left(\frac{1}{4e}\right)^{\frac{1}{3}}-\frac{u}{\varepsilon}\right)\\
	&>B^{\frac{1}{3}}\left((|a|b(a+b)e^{-1})^{\frac{1}{3}}-\varepsilon^{-1}ub\right).
\end{align*}
Donc l'intervalle $I$ (\ref{eq:interval}) ne contient qu'au plus un point du réseau $\Gamma(a,b)$.  
Nous allons donc décomposer la somme (\ref{eq:sommeF}) en deux termes selon la taille de $a$ et $b$:
\begin{equation}\label{eq:decomp}
\varrho_{B,Q}(\chi(\varepsilon))=\sum_{u,e_i}\left(\sum_{\substack{\pgcd(a,b)=1\\ e_1|a,e_2|b,e_3|a+b\\b\leqslant  DB^{\frac{1}{6}}}}+\sum_{\substack{\pgcd(a,b)=1\\  e_1|a,e_2|b,e_3|a+b\\b> DB^{\frac{1}{6}}}}\right)\sharp F_{u,e_i,\varepsilon}(a,b).
\end{equation}

\textbf{Cas I}. $b\leqslant DB^{\frac{1}{6}}$. La longueur de $I$ est \og assez grande\fg.
La contribution de ce cas est la suivante.

\begin{propo}\label{po:globalcasi}
 \begin{equation}
  \sum_{\substack{\pgcd(a,b)=1\\  e_1|a,e_2|b,e_3|a+b\\b\leqslant DB^{\frac{1}{6}}} }
\sharp F_{u,e_i,\varepsilon}(a,b)=\frac{Z(\varepsilon,u,e)}{\pi^2}B^{\frac{1}{3}}\log B + O(B^{\frac{1}{3}}\log \log B),
 \end{equation}
où $Z(\varepsilon,u,e)$ est donné par (\ref{eq:Zeue}) et $F_{u,e_i,\varepsilon}(a,b)$ est défini par (\ref{eq:conditionaa}) et (\ref{eq:conditionbb}).
\end{propo}
On a besoin du lemme suivant qui nous permet de changer les sommes sur un réseau en intégrales.
\begin{lemme}\label{le:echange}
 Pour $DB^{\frac{1}{6}}\geqslant M\gg e$ et $\underline{k}=(k_1,k_2,k_3)$ fixé, on note $k=k_1k_2k_3,\alpha_i=k_i e_i$ et $\alpha=\alpha_1\alpha_2\alpha_3$.
 Considérons deux régions
 $$S^\prime(M)=S^\prime(e,u,\varepsilon,M,B)=\{(m,n)\in \mathbb{Z}^2: 1\leqslant m<n\leqslant M,\frac{m}{n}\in \mathopen ]\frac{1}{2}-C_1,\frac{1}{2}+C_1\mathclose [\},$$
  $$T(M)=T(u,e,\varepsilon,M)=\{(x,y)\in\mathbb{R}^2:1\leqslant y< x\leqslant M, \frac{y}{x}\in \mathopen ]\frac{1}{2}-C_0,\frac{1}{2}+C_0\mathclose [\},$$
 avec $C_1=C_1(m,n,u,e,\varepsilon)$ dépendant de $m,n$ définie par l'équation (\ref{eq:C1}) et $C_0=C_0(u,e,\varepsilon)$ par (\ref{eq:c0}).
 Alors on a
 $$\sum_{(m,n)\in S^\prime(M)\cap H(\alpha_i)}
 (|m|n(m+n))^{-\frac{2}{3}}=\iint_{T(M)}
 \frac{\operatorname{d}x\operatorname{d}y}{\alpha (xy(x-y))^{\frac{2}{3}}}+O_{u,e,\alpha,\varepsilon}(\log \log B),$$
 $$\sum_{(m,n)\in S^\prime(M)\cap H(\alpha_i)}
 (m(n-m))^{-1}=\iint_{T(M)}
 \frac{\operatorname{d}x\operatorname{d}y}{\alpha x y(x-y)}+O_{u,e,\alpha,\varepsilon}(\log \log B).$$
 \end{lemme}
\begin{proof}[Démonstration du lemme]
 L'idée est qu'on fait une partition sur le domaine de $(m,n)$.

 On définit un sous-ensemble de $\mathbb{N}^2$ auxiliaire
 $$S(M)=\{(m,n)\in \mathbb{Z}^2: 1\leqslant m<n\leqslant M,\frac{m}{n}\in \mathopen ]\frac{1}{2}-C_0,\frac{1}{2}+C_0\mathclose [\}.$$
 La première étape est de comparer la somme sur $S(M)$ et $S^\prime(M)$.
 Pour $(m,n)\in S(M)$, on note $C_2=C_2(m,n,u,e,\varepsilon)=C_0-C_1$.
 Rappelons, d'après (\ref{eq:C1small}), que l'on a $C_2=O_{u,e,\varepsilon}(B^{-\frac{1}{3}})$.
La différence entre la somme sur $S(M)$ et $S^\prime(M)$ s'écrit
 \begin{align*}
 &\sum_{(m,n)\in S(M)\cap H(\alpha_i)}
 (mn(n-m))^{-\frac{2}{3}}-\sum_{(m,n)\in S^\prime(M)\cap H(\alpha_i)}
 (mn(n-m))^{-\frac{2}{3}}\\
 &\leqslant \sum_{\substack{n\leqslant M\\ \frac{m}{n}\in \mathopen ]\frac{1}{2}-C_0,\frac{1}{2}-C_1 \mathclose ] \\
 \cup  \mathclose [\frac{1}{2}+C_1,\frac{1}{2}+C_0 \mathclose ]}}
 (mn(n-m))^{-\frac{2}{3}}\\
&=\sum_{n\leqslant M}\frac{1}{n^2}\sum_{\substack{(\frac{1}{2}-C_0)n<m\leqslant(\frac{1}{2}-C_1)n\\ 
\mbox{ou } (\frac{1}{2}+C_1)n\leqslant m<(\frac{1}{2}+C_0)n}}\left(\frac{m}{n}\left(1-\frac{m}{n}\right)\right)^{-\frac{2}{3}}.
 \end{align*}
 Or pour $(m,n)\in S(M)$, on a $n\leqslant M\leqslant DB^{\frac{1}{6}}$. Donc pour $B$ assez grand, d'après (\ref{eq:C1small}),
 $$C_2 n=\left(\frac{1}{2}-C_1\right)n-\left(\frac{1}{2}-C_0\right)n=\left(\frac{1}{2}+C_0\right)n-\left(\frac{1}{2}+C_1\right)n=O_{u,e,\varepsilon}(B^{-\frac{1}{6}}).$$
 Donc pour $B$ suffisamment grand, $C_2 n<1$.
 Cela nous dit qu'ayant fixé $n$, il y a au plus deux $m$ possibles dans la somme. On obtient la majoration suivante pour
 la dernière somme ci-dessus.
 \begin{align*}
 &\sum_{n\leqslant M}\frac{1}{n^2}\sum_{\substack{(\frac{1}{2}-C_0)n<m\leqslant(\frac{1}{2}-C_1)n\\ 
\mbox{ou } (\frac{1}{2}+C_1)n\leqslant m<(\frac{1}{2}+C_0)n}}\left(\frac{m}{n}\left(1-\frac{m}{n}\right)\right)^{-\frac{2}{3}}\\
&= O_{u,e,\varepsilon}(1) \sum_{n\in\mathbb{N}}\frac{1}{n^2}=O_{u,e,\varepsilon}(1).
 \end{align*}
 Ensuite on fixe une constante $$l=l(\alpha)=\log_{\frac{3}{2}}(\log_\alpha M)=\frac{\log \log M-\log\log \alpha}{\log \left(\frac{3}{2}\right)}. $$
 Alors on a $M^{\left(\frac{2}{ 3}\right)^l}=\alpha$. Pour $0\leqslant k\leqslant l-1$, on note
 $$F_{k,\alpha_i}=\sum _{\substack{(m,n)\in S(M)\cap H(\alpha_i)\\M^{\left(\frac{2}{ 3}\right)^{k+1}}<n\leqslant M^{\left(\frac{2}{ 3}\right)^k}}}
 (mn(n-m))^{-\frac{2}{3}},$$
 $$G_{\alpha_i}=\sum _{\substack{(m,n)\in S(M)\cap H(\alpha_i)\\n\leqslant \alpha}}(mn(n-m))^{-\frac{2}{3}}.$$
 On décompose la somme de la façon suivante.
 \begin{align*}
  &\sum_{(m,n)\in S(M)\cap H(\alpha_i)}
 (mn(n-m))^{-\frac{2}{3}}\\
 &=\sum_{0\leqslant k\leqslant l-1}\sum _{\substack{(m,n)\in S(M)\cap H(\alpha_i)\\M^{\left(\frac{2}{ 3}\right)^{k+1}}<n\leqslant M^{\left(\frac{2}{ 3}\right)^k}}}
 (mn(n-m))^{-\frac{2}{3}}+\sum _{\substack{(m,n)\in S(M)\cap H(\alpha_i)\\n\leqslant \alpha}}(mn(n-m))^{-\frac{2}{3}}\\
 &=\sum_{0\leqslant k\leqslant l-1} F_{k,\alpha_i}(M)+G_{\alpha_i} (M).
 \end{align*}
 Premièrement on a $$G_{\alpha_i}(M)=O_{\alpha,e,u,\varepsilon} (1).$$
 Ensuite on compare chaque morceau $F_{k,e}(M)$ avec l'intégrale.
 Pour cela on va fixer les domaines fondamentaux du réseau $H(\alpha_i)$.
 On peut choisir une base $(e_1(\alpha_i),e_2(\alpha_i))$ engendrant $H(\alpha_i)$ telle que 
 $$\|e_j(\alpha_i)\|\leqslant 2\lambda_j(\alpha_i),\quad (1\leqslant j\leqslant 2)$$
 où $\lambda _j$ désigne le $j$-ième successif minima par rapport à la norme euclidienne standard $\|\cdot\|$ (voir, par exemple, \cite{cassels}, p.135).
 Les termes d'erreur viennent des bords de notre partition et du passage à l'intégrale sur un domaine fondamental.
 On note $\mathbb{B}(M,k)$ la réunion des domaines fondamentaux du réseau $H(\alpha_i)$
 dont l'intersection avec le bord du domaine
 \begin{equation}\label{eq:smalldomain}
 \mathbb{D}(M,k)=\{(x,y)\in T(M):M^{\left(\frac{2}{ 3}\right)^{k+1}}<x\leqslant M^{\left(\frac{2}{ 3}\right)^k}\}
 \end{equation}
 est non-vide.
 On a que
 \begin{align*}\sup_{(x,y)\in\mathbb{B}(M,k)\cap \mathbb{D}(M,k)}(xy(x-y))^{-\frac{2}{3}}
 &\leqslant \sup_{(x,y)\in\mathbb{B}(M,k)\cap \mathbb{D}(M,k)}\left(\frac{y}{x}\left(1-\frac{y}{x}\right)\right)^{-\frac{2}{3}}
 \sup_{M^{\left(\frac{2}{ 3}\right)^{k+1}}<x\leqslant M^{\left(\frac{2}{ 3}\right)^k}}\frac{1}{x^2}\\
 &=O_{u,e,\varepsilon}(1)\sup_{M^{\left(\frac{2}{ 3}\right)^{k+1}}<x\leqslant M^{\left(\frac{2}{ 3}\right)^k}}\frac{1}{x^2}\\
 &=O_{u,e,\varepsilon}(M^{-2\left(\frac{2}{ 3}\right)^{k+1}}).\end{align*}
 Maintenant on fixe $(m,n)\in S(M)\cap H(\alpha_i)$ avec $M^{\left(\frac{2}{ 3}\right)^{k+1}}<n\leqslant M^{\left(\frac{2}{ 3}\right)^k}$.
 On prend un domaine fondamental $\mathcal{F}$ contenant $(m,n)$.
 Soit $(x,y)\in \mathcal{F}$. Alors
 $$(x,y)=(m,n)+se_1+te_2,$$
 avec $s,t\in\mathopen[-1,1\mathclose]$. On note $\lambda(\alpha_i)=2(\lambda_1(\alpha_i)+\lambda_2(\alpha_i))$. D'après l'inégalité de Minkowski, on a
 $$\lambda
 (\alpha_i)\leqslant 4\lambda_1(\alpha_i)\leqslant 8\sqrt{\frac{\operatorname{det}(H(\alpha_i))}{\operatorname{vol}(\mathbf{B}(0,1))}}=\frac{8}{\sqrt{\operatorname{vol}(\mathbf{B}(0,1))}}\sqrt{\alpha_1\alpha_2\alpha_3},$$
 où $\operatorname{vol}(\mathbf{B}(0,1))$ est le volume de la boule unité par rapport à la norme choisie. Donc $$|x-m|\leqslant \|e_1(\alpha_i)\|+\|e_2(\alpha_i)\|\leqslant 2(\lambda_1(\alpha_i)+\lambda_2(\alpha_i))= \lambda(\alpha_i)=O_{\alpha_i}(1),$$ 
 $$|y-n|=O_{\alpha_i}(1).$$
 Comme pour $n>\lambda(\alpha_i)$, 
 $$\frac{m-\lambda(\alpha_i)}{n+ \lambda(\alpha_i)} < \frac{y}{x}< \frac{m+ \lambda(\alpha_i)}{n-\lambda(\alpha_i)},$$
 on a
 $$\left|\frac{y}{x}-\frac{m}{n}\right|\leqslant\max \left( \frac{m}{n}-\frac{m}{n+ \lambda(\alpha_i)},\frac{m+ \lambda(\alpha_i)}{n}- \frac{m}{n}\right)
 =O_{u,e,\varepsilon,\alpha_i}\left(\frac{1}{n}\right).$$
 Considérons la fonction suivante définie sur $]\frac{1}{2}-C_0,\frac{1}{2}+C_0\mathclose [$:
 $$\lambda\longmapsto (\lambda(1-\lambda))^{-\frac{2}{3}}.$$
 En reportant les majorations ci-dessus et en utilisant le théorème de la valeur moyenne, on obtient que
 $$\left|\left(\frac{y}{x}\left(1-\frac{y}{x}\right)\right)^{-\frac{2}{3}}-\left(\frac{m}{n}\left(1-\frac{m}{n}\right)\right)^{-\frac{2}{3}}\right|= O_{u,e,\varepsilon,\alpha}\left(\frac{1}{n}\right).$$
 On a en outre que
 \begin{align*}
  \left|\frac{1}{x^2}-\frac{1}{n^2}\right|=\frac{(x-n)(x+n)}{x^2 n^2}\leqslant \frac{2\lambda(\alpha_i) (n+\lambda(\alpha_i))}{n^4}= O_{\alpha}(M^{-3\left(\frac{2}{ 3}\right)^{k+1}}).
 \end{align*}
Tout cela nous fournit la majoration suivante.
 \begin{align*}
  &|(xy(x-y))^{-\frac{2}{3}}-(mn(n-m))^{-\frac{2}{3}}|\\
  =&\left|\frac{1}{x^2}\left(1-\frac{y}{x}\right)^{-\frac{2}{3}}-\frac{1}{n^2}\left(1-\frac{m}{n}\right)^{-\frac{2}{3}}\right|\\  
  \leqslant &\left|\frac{1}{x^2}\left(\frac{y}{x}\left(1-\frac{y}{x}\right)\right)^{-\frac{2}{3}}-
  \frac{1}{n^2}\left(\frac{y}{x}\left(1-\frac{y}{x}\right)\right)^{-\frac{2}{3}}\right|
  + \left|\frac{1}{n^2}\left(\frac{y}{x}\left(1-\frac{y}{x}\right)\right)^{-\frac{2}{3}}-
  \frac{1}{n^2}\left(\frac{m}{n}\left(1-\frac{m}{n}\right)\right)^{-\frac{2}{3}}\right|\\
=& O_{u,e,\varepsilon,\alpha}(1) \left( \frac{1}{x^2}-\frac{1}{n^2}\right)+O_{u,e,\varepsilon,\alpha}\left(\frac{1}{n^3}\right)\\
=& O_{u,e,\varepsilon,\alpha}(M^{-3\left(\frac{2}{ 3}\right)^{k+1}}).
 \end{align*}
 En remarquant que le périmètre du bord et l'aire du domaine $\mathbb{D}(M,k)$ (\ref{eq:smalldomain}) sont de grandeur 
 $O_{u,e,\varepsilon,\alpha}(M^{\left(\frac{2}{ 3}\right)^{k}})$ et $O_{u,e,\varepsilon,\alpha}(M^{2\left(\frac{2}{ 3}\right)^{k}})$ respectivement,
 il suit des estimations ci-dessus que
 \begin{align*}
 &\left|F_{k,\alpha_i}(M)-\iint_{\substack{(x,y)\in T(M) \\M^{\left(\frac{2}{ 3}\right)^{k+1}}<x\leqslant M^{\left(\frac{2}{ 3}\right)^k}}}
 \frac{\operatorname{d}x\operatorname{d}y}{\alpha(xy(x-y))^{\frac{2}{3}}}\right|\\
 &\ll\left(\sup_{(x,y)\in\mathbb{B}(M,k)\cap \mathbb{D}(M,k)}
 (xy(x-y))^{-\frac{2}{3}} \right)O_{\alpha,e,u,\varepsilon}(M^{\left(\frac{2}{ 3}\right)^{k}})\\
 &\ +O_{u,e,\varepsilon,\alpha}(M^{-3\left(\frac{2}{ 3}\right)^{k+1}})
 \iint_{\substack{(x,y)\in T(M) \\M^{\left(\frac{2}{ 3}\right)^{k+1}}<x\leqslant M^{\left(\frac{2}{ 3}\right)^k}}}\operatorname{d}x\operatorname{d}y\\
 &= O_{u,e,\varepsilon,\alpha}(M^{-\frac{1}{3}\left(\frac{2}{ 3}\right)^{k}})+O_{u,e,\varepsilon,\alpha}(1)\\
 &= O_{\alpha,e,u,\varepsilon}(1).
 \end{align*}
En reportant de ce que nous avons obtenu,  
 \begin{align*}
 &\sum_{(m,n)\in S(M)\cap H(\alpha_i)}
 (mn(n-m))^{-\frac{2}{3}}\\
 &=\sum_{0\leqslant k\leqslant l-1} F_{k,\alpha_i}(M)+G_{\alpha_i} (M)\\
 &=\sum_{0\leqslant k\leqslant l-1}\iint_{\substack{(x,y)\in T(M) \\M^{\left(\frac{2}{ 3}\right)^{k+1}}<x\leqslant M^{\left(\frac{2}{ 3}\right)^k}}}
 \frac{\operatorname{d}x\operatorname{d}y}{\alpha(xy(x-y))^{\frac{2}{3}}}+O_{\alpha,e,u,\varepsilon}(1)\sum_{0\leqslant k\leqslant l} 1 +O_{\alpha,e,u,\varepsilon}(1)\\
 &=\iint_{T(M)}
 \frac{\operatorname{d}x\operatorname{d}y}{\alpha(xy(x-y))^{\frac{2}{3}}}
 +O_{u,\alpha,e,\varepsilon}(\log \log M)\\
 &=\iint_{T(M)}
 \frac{\operatorname{d}x\operatorname{d}y}{\alpha(xy(x-y))^{\frac{2}{3}}}
 +O_{u,\alpha,e,\varepsilon}(\log \log B).
\end{align*}
La seconde égalité se démontre exactement de la même façon,
d'où le lemme.
\end{proof}

On introduit la fonction arithmétique $\psi$ donnée par
\begin{equation}\label{eq:thefunctionpsi}
\psi(n)=\operatorname{Card} K(n)
\end{equation}
où $$K(n)=\{(e_1,e_2,e_3)\in\mathbb{N}^3:n=e_1e_2e_3,\pgcd(e_i,e_j)=1 (\forall i\neq j)\}.$$
\begin{propo}\label{po:thefunctionpsi}
 La fonction $\psi$ est multiplicative. On a $\psi(n)=3^{\omega(n)}$.
\end{propo}
\begin{proof}[Démonstration]
 On se donne $m,n\in\mathbb{N}$ avec $\pgcd(m,n)=1$. L'application
 $$K(m)\times K(n)\longrightarrow K(mn)$$
 $$(e_1,e_2,e_3)\times (f_1,f_2,f_3)\mapsto(e_1f_1,e_2f_2,e_3f_3)$$
 est une bijection.
 Donc $$\psi(n)=\prod_{p|n} \psi(p^{v_p(n)})=\prod_{p|n} 3=3^{\omega(n)}.$$
\end{proof}

\begin{proof}[Démonstration de la proposition \ref{po:globalcasi}]
On suppose que
\begin{equation}\label{eq:conditioncc}
\pgcd\left(u,\frac{ab(a+b)}{e}\right)=1.
\end{equation}
 On utilise la formule classique (voir \cite{Browning}, Exercise 5.2) pour l'intervalle (\ref{eq:interval}), compte tenu des conditions \textcircled{a} (\ref{eq:conditionaa}),\textcircled{b} (\ref{eq:conditionbb}) et du fait que la condition \textcircled{b} équivaut à la condition de l'ensemble $\Gamma^\prime(a,b)$ (\ref{eq:latticecondition}) plus la condition (\ref{eq:conditioncc}), 
 \begin{equation}\label{eq:classique}
 \begin{split}
  \operatorname{Card}(F_{u,e_i,\varepsilon}(a,b))&=\frac{B^{\frac{1}{3}}((e^{-1}|a|b(a+b))^{\frac{1}{3}}-ub\varepsilon^{-1})}{e^{-1}|a|b(a+b)}\phi(eu)+O(2^{\omega(eu)})+O_{u,e}(1)\\
  &=B^{\frac{1}{3}}\phi(eu)\left(e^{\frac{2}{3}}(|a|b(a+b))^{-\frac{2}{3}}-eu\varepsilon^{-1}(|a|(a+b))^{-1}\right)+O_{u,e}(1).
 \end{split}\end{equation}
 Le terme d'erreur peut être contrôlé par
 $$\sum_{\substack{
   \pgcd(a,b)=1, b\leqslant DB^{\frac{1}{6}}\\
   F_{u,e_i,\varepsilon}(a,b)\neq \varnothing  
}}1\leqslant \sum_{\substack{
    b\leqslant DB^{\frac{1}{6}}\\
    |\frac{a}{b}|\in \mathopen ]\frac{1}{2}-C_0,\frac{1}{2}+C_0\mathclose [
}} 1=O_{u,e,\varepsilon}(B^{\frac{1}{3}}).$$

On réécrit la somme dans la Proposition \ref{po:globalcasi} pour éliminer la condition de pgcd (\ref{eq:conditioncc}).

\begin{align*}
 &\sum_{\substack{
   \pgcd(a,b)=1,b\leqslant DB^{\frac{1}{6}}\\
   (a,b)\in H(e_i) ,\pgcd(u,\frac{ab(a+b)}{e})=1\\\left|\frac{a}{b}\right|\in\mathopen [\frac{1}{2}-C_1,\frac{1}{2}+C_1\mathclose ]
}}(|a|b(a+b))^{-\frac{2}{3}}\\
&=\sum_{k|u}\mu(k)\sum_{\substack{k_1k_2k_3=k\\ \pgcd(k_i,k_j)=1}}\sum_{\substack{\pgcd(a,b)=1,b\leqslant DB^{\frac{1}{6}}\\ e_1k_1|a,e_2k_2|b,e_3k_3|a+b\\\left|\frac{a}{b}\right|\in\mathopen [\frac{1}{2}-C_1,\frac{1}{2}+C_1\mathclose ]
}}(|a|b(a+b))^{-\frac{2}{3}}.
\end{align*}
 Pour $k_1,k_2,k_3$ fixés,
 \begin{equation}\label{eq:somme K}
\begin{split}
&\sum_{\substack{\pgcd(a,b)=1,b\leqslant DB^{\frac{1}{6}}\\ e_1k_1|a,e_2k_2|b,e_3k_3|a+b\\\left|\frac{a}{b}\right|\in\mathopen [\frac{1}{2}-C_1,\frac{1}{2}+C_1\mathclose ]
}}(|a|b(a+b))^{-\frac{2}{3}}\\
&=\sum_{d\in\mathbb{N}}\mu(d)\sum_{\substack{d|a,d|b,b\leqslant DB^{\frac{1}{6}}\\(a,b)\in H(e_ik_i)\\|\frac{a}{b}|\in \mathopen [\frac{1}{2}-C_1,\frac{1}{2}+C_1\mathclose ]}}
(|a|b(a+b))^{-\frac{2}{3}}\\
&=\left(\sum_{d\leqslant D(ek)^{-1}B^{\frac{1}{6}}}+\sum_{d> D(ek)^{-1}B^{\frac{1}{6}}}\right)\frac{\mu(d)}{d^2}\sum_{{\substack{1\leqslant a< b\leqslant DB^{\frac{1}{6}}d^{-1}\\ |\frac{a}{b}|\in 
\mathopen [\frac{1}{2}-C_1,\frac{1}{2}+C_1\mathclose ]\\(a,b)\in H(e_ik_i) }}}
(|a|b(a+b))^{-\frac{2}{3}}.
\end{split} \end{equation}
La seconde somme est bornée. En écrivant \og$O$\fg\  pour \og$O_{u,e,k,\varepsilon}$\fg, le Lemme \ref{le:echange} nous fournit que la somme ci-dessus est égale à 
\begin{align*}
&\sum_{d\leqslant D(ek)^{-1}B^{\frac{1}{6}}} \frac{\mu(d)}{d^2}\left(\iint_{T(DB^{\frac{1}{6}}d^{-1})}
 \frac{\operatorname{d}x\operatorname{d}y}{ek(xy(x-y))^{\frac{2}{3}}}+O(\log \log B)\right)+O(1)\\
 &=\sum_{d\leqslant D(ek)^{-1}B^{\frac{1}{6}}}\frac{\mu(d)}{ek d^2} \iint_{\substack{1\leqslant x\leqslant DB^{\frac{1}{6}}d^{-1}\\\theta\in \mathopen ]\frac{1}{2}-C_0,\frac{1}{2}+C_0\mathclose [}}
 \frac{\operatorname{d}\theta\operatorname{d}x}{x(\theta(1-\theta))^{\frac{2}{3}}}+O(\log \log B)\\
 &=\sum_{d\leqslant D(ek)^{-1}B^{\frac{1}{6}}}\frac{\mu(d)}{ek d^2}\left(\frac{1}{6}\log B+\log D-\log d\right)\int_{\theta\in \mathopen ]\frac{1}{2}-C_0,\frac{1}{2}+C_0\mathclose [} 
 \frac{\operatorname{d}\theta}{(\theta(1-\theta))^{\frac{2}{3}}}+O(\log \log B).\\
 \end{align*}
 Comme précédemment on peut rajouter une quantité bornée pour obtenir une somme où $d$ parcourt l'ensemble des entiers naturels.
 Puisque l'on a des formules bien connues
$$\sum_{d\in\mathbb{N}}\frac{\log d}{d^2}<\infty,\quad \sum_{d\in\mathbb{N}}\frac{\mu(d)}{d^2}=\frac{6}{\pi^2},$$
la somme (\ref{eq:somme K}) est finalement de la forme
\begin{align*}
 &=\sum_{d\in\mathbb{N}}\frac{\mu(d)}{ek d^2}\frac{1}{6}\log B\int_{\theta\in \mathopen ]\frac{1}{2}-C_0,\frac{1}{2}+C_0\mathclose [} \frac{\operatorname{d}\theta}{(\theta(1-\theta))^{\frac{2}{3}}}+O(\log \log B)\\
&= (ek)^{-1}\pi^{-2}\int_{\theta\in \mathopen ]\frac{1}{2}-C_0,\frac{1}{2}+C_0\mathclose [} \frac{\operatorname{d}\theta}{(\theta(1-\theta))^{\frac{2}{3}}}\log B +O(\log \log B).
\end{align*}

De la même façon,
\begin{align*}
 &\sum_{\substack{
    \pgcd(a,b)=1\\ b\leqslant DB^{\frac{1}{6}}\\
   (a,b)\in H(e_ik_i) 
}}(|a|(a+b))^{-1}\\
&=\frac{1}{ek\pi^2}\int_{\theta\in \mathopen ]\frac{1}{2}-C_0,\frac{1}{2}+C_0\mathclose [} \frac{\operatorname{d}\theta}{(\theta(1-\theta))}\log B +O(\log \log B)
\end{align*}

Il en résulte que, en rappelant la définition de la fonction $\psi$ \eqref{eq:thefunctionpsi} et la Proposition \ref{po:thefunctionpsi},
\begin{align*}
 &\sum_{\substack{
   \pgcd(a,b)=1,b\leqslant DB^{\frac{1}{6}}\\
   (a,b)\in H(e_i) ,\pgcd(u,\frac{ab(a+b)}{e})=1\\\left|\frac{a}{b}\right|\in\mathopen [\frac{1}{2}-C_1,\frac{1}{2}+C_1\mathclose ]
}}\sharp F_{u,e_i,\varepsilon}(a,b)\\
&=B^{\frac{1}{3}}\phi(eu)\left(\sum_{k|u} \mu(k)\psi(k)\sum_{\substack{
   \pgcd(a,b)=1,b\leqslant DB^{\frac{1}{6}}\\
   (a,b)\in H(e_ik_i) \\\left|\frac{a}{b}\right|\in\mathopen [\frac{1}{2}-C_1,\frac{1}{2}+C_1\mathclose ]
}}\left(e^{\frac{2}{3}}(|a|b(a+b))^{-\frac{2}{3}}-eu\varepsilon^{-1}(|a|(a+b))^{-1}\right)\right)+O(B^{\frac{1}{3}})\\
&=B^{\frac{1}{3}}\log B\left(\phi(eu)\sum_{k|u}\mu(k)3^{\omega(k)}\left(\frac{1}{e^{\frac{1}{3}}k\pi^2}\int_{\theta\in \mathopen ]\frac{1}{2}-C_0,\frac{1}{2}+C_0\mathclose [} \frac{\operatorname{d}\theta}{(\theta(1-\theta))^{\frac{2}{3}}}-\frac{u}{\varepsilon k \pi^2}\int_{\theta\in \mathopen ]\frac{1}{2}-C_0,\frac{1}{2}+C_0\mathclose [} \frac{\operatorname{d}\theta}{(\theta(1-\theta))}\right)\right)\\&+O(B^{\frac{1}{3}}\log\log B)\\
&=\frac{Z(\varepsilon,u,e)}{\pi^2}B^{\frac{1}{3}}\log B + O(B^{\frac{1}{3}}\log \log B),
\end{align*}
ce qui clôt la démonstration.
\end{proof}
On termine en remarquant que le terme principal ne dépend pas de la constante $D$. 
Dans la suite, nous utiliseront le lemme suivant qui montre directement cette indépendance.
\begin{lemme} \label{le:constante}
Fixons deux constantes $C_1,C_2>0$, on a
$$\sum_{\substack{\pgcd(a,b)=1\\C_2 B^{\frac{1}{6}}\leqslant b\leqslant C_1 B^{\frac{1}{6}} }}\sharp  F_{u,e_i,\varepsilon}(a,b)=O_{C_1,C_2,u,e,\varepsilon}(B^{\frac{1}{3}}).$$  \end{lemme}
\begin{proof}[Démonstration]
 En fait ce lemme découle du calcul précédent dont le terme d'erreur a la grandeur $O(B^{\frac{1}{3}}\log\log B)$, qui est acceptable pour la suite. On donne une preuve directe avec la grandeur de contrôle souhaitée.
 D'après la formule (\ref{eq:classique}), il suffit de contrôler la somme sur les quantités $(|a|b(a+b))^{-\frac{2}{3}}$ et $ (|a|(a+b))^{-1}$.
On a
 \begin{align*}
   &\sum_{\substack{C_2 B^{\frac{1}{6}}\leqslant b\leqslant C_1 B^{\frac{1}{6}}\\|\frac{a}{b}|\in ]\frac{1}{2}-C_0,\frac{1}{2}+C_0[}}
(|a|b(a+b))^{-\frac{2}{3}}\\
&\leqslant 
\sup_{\substack{C_2 B^{\frac{1}{6}}\leqslant b\leqslant C_1 B^{\frac{1}{6}}\\|\frac{a}{b}|\in \mathopen ]\frac{1}{2}-C_0,\frac{1}{2}+C_0\mathclose [}}
 (|a|b(a+b))^{-\frac{2}{3}} \sum_{\substack{C_2 B^{\frac{1}{6}}\leqslant b\leqslant C_1 B^{\frac{1}{6}}\\|\frac{a}{b}|\in \mathopen ]\frac{1}{2}-C_0,\frac{1}{2}+C_0\mathclose [}} 1\\
 &=O_{C_1,C_2}(B^{-\frac{1}{3}})O_{C_1,C_2}(B^{\frac{1}{3}})=O_{C_1,C_2}(1).
 \end{align*}
 De façon similaire, 
   $$\sum_{\substack{C_2 B^{\frac{1}{6}}\leqslant b\leqslant C_1 B^{\frac{1}{6}}\\|\frac{a}{b}|\in \mathopen ]\frac{1}{2}-C_0,\frac{1}{2}+C_0\mathclose [}}
   (|a|(a+b))^{-1}=O_{C_1,C_2}(1).$$
  Notons que le nombre de $(a,b)$ possibles est $O(B^{\frac{1}{3}})$. On obtient le résultat souhaité.
\end{proof}

\textbf{Cas II}. $b\geqslant DB^{\frac{1}{6}}$.

Dans ce cas là, le terme principal et le terme d'erreur ont le même ordre de grandeur
(du point de vue de la démonstration de la Proposition \ref{po:globalcasi}).  
Mais grâce à la symétrie de la transformation de Cremona on peut se ramener presque totalement au cas précédent. Dans la section 3.2 on a défini les paramètres correspondant aux points dans $S$ en échangeant $x$ avec $z$. De façon similaire on peut fixer le difféomorphisme $\tilde{\rho}$ donné par
$$[x:y:z]\longmapsto\left(\frac{z}{x}-1,\frac{y}{x}-1\right)\in\mathbb{R}^2,$$ qui s'identifie $S$ avec son image dans $T_Q X_3$. On redéfinit l'ensemble $F(u,e,\varepsilon)$ (\ref{eq:coutingwithfixedparameters}) pour la région $S$ (\ref{regions}) comme suit:
   \begin{align}\label{eq:coutingwithfixedparameters2}
   F^\prime(u,e_i,\varepsilon) &=\left\{
   \begin{aligned}
   &P=[x:y:z]\\
   &z>y>x>0\\
   &\pgcd(x,y,z)=1
   \end{aligned}
   \left|
   \begin{aligned}
   &\left|\frac{z}{x}-1\right|\leqslant \varepsilon B^{-\frac{1}{3}}, u(P)=u,\mathbf{e}(P)=(e_3,e_2,e_1);\\
   &\frac{\max(|x|,|y|)\max(|x|,|z|)\max(|y|,|z|)}{\pgcd(x,y)\pgcd(x,z)\pgcd(x,z)}\leqslant B
   \end{aligned}
   \right\}\right.\end{align}
   D'après les Propositions \ref{po:Cremona1} et \ref{po:Cremona2}, la transformation de Cremona $\Psi$ échange les points dans $F(u,e_i,\varepsilon)$ avec ceux dans $F^\prime(u,e_i,\varepsilon)$. Et symétriquement on a des résultats analogues à ceux du début de la Section 4. En particulier (voir (\ref{eq:slope})), pour $U\in F^\prime(u,e_i,\varepsilon)$, on a
   $$\mu^\prime(U)=\frac{m^\prime(U)}{n^\prime(U)}\in \left] \frac{1}{2}-C_0,\frac{1}{2}+C_0\right[.$$
\begin{lemme}\label{le:Cremona}
 On note $$E=E(\varepsilon,u,e)=\frac{e^{\frac{2}{3}}}{(\frac{1}{2}-C_0)^2 4^{\frac{1}{3}}D},
 \quad G=G(\varepsilon,e,u)=\frac{ue}{\varepsilon(\frac{1}{2}+C_0)^2D},$$ 
$$A_1=\{V\in F(u,e_i,\varepsilon):n(V)>DB^{\frac{1}{6}}\},$$
$$A_2^\prime=\{U\in  F^\prime(u,e_i,\varepsilon): n^\prime(U)\leqslant EB^{\frac{1}{6}}\}, $$
$$A_3^\prime=\{U\in  F^\prime(u,e_i,\varepsilon): n^\prime(U)\leqslant GB^{\frac{1}{6}}\}.$$
Alors on a pour $B\gg_{u,e,\varepsilon} 1$,
 \begin{equation}
   A_3^\prime\subseteq \Psi(A_1)\subseteq A_2^\prime,
 \end{equation}
\end{lemme}

\begin{proof}[Démonstration]

 On note $\lambda=-\frac{a}{b}(1+\frac{a}{b})$.
 Fixons $V\in A_1$ avec ses paramètres $a,b,v,u$ et notons $U=\Psi(V)\in F^\prime(u,e_i,\varepsilon)$.
 D'après la Proposition (\ref{po:Cremona2}), 
  \begin{equation}
   m^\prime(U)=-\frac{e_1e_2e_3v}{ab},\quad n^\prime(U)=\frac{e_1e_2e_3(-ua+v)}{-a(a+b)}.
  \end{equation}
 La condition $\textcircled{a}$ (\ref{eq:conditionaa}) fournit que
 \begin{equation}
  \frac{uB^{\frac{1}{3}}}{\varepsilon}\leqslant\frac{v}{b}<\frac{B^{\frac{1}{3}}\lambda^{\frac{1}{3}}}{e^\frac{1}{3}},
 \end{equation}
 d'où
 \begin{equation}\label{eq:relationpente}
 \frac{ueB^{\frac{1}{3}}}{|a|\varepsilon}\leqslant m^\prime(U) <\frac{e^{\frac{2}{3}}B^{\frac{1}{3}}\lambda^{\frac{1}{3}}}{|a|}.
 \end{equation}
 et donc 
 $$m^\prime(U)<\frac{e^{\frac{2}{3}}B^{\frac{1}{3}}}{(\frac{1}{2}-C_0)4^{\frac{1}{3}}b}\leqslant 
 \frac{e^{\frac{2}{3} }B^{\frac{1}{6}}}{(\frac{1}{2}-C_0)4^{\frac{1}{3}}D},$$
 $$n^\prime(U)<\frac{m^\prime(U)}{\frac{1}{2}-C_0}<\frac{e^{\frac{2}{3} }B^{\frac{1}{6}}}{(\frac{1}{2}-C_0)^24^{\frac{1}{3}}D}=EB^{\frac{1}{6}}$$
 Cela montre que 
 \begin{align*}
  \Psi(A_1)\subseteq A_2^\prime.
 \end{align*}
En revanche, 
 pour tout $U\in A_3^\prime$, on a de la même façon que (\ref{eq:relationpente}),
 $$m(\Psi(U))\geqslant \frac{ueB^{\frac{1}{3}}}{\varepsilon m^\prime(U)}\geqslant \frac{ueB^{\frac{1}{6}}}{\varepsilon(\frac{1}{2}+C_0) G},$$
 et donc
 $$n(\Psi(U))> \frac{m(\Psi(U))}{\frac{1}{2}+C_0}> \frac{ueB^{\frac{1}{6}}}{\varepsilon(\frac{1}{2}+C_0)^2G}=DB^{\frac{1}{6}}.$$
 Cela montre que
 $$\Psi(A_3^\prime)\subseteq A_1.$$
 Puisque $\Psi^2=$ Id, on obtient que
 $$A_3^\prime\subseteq\Psi(A_1).$$
\end{proof}

On peut maintenant prouver notre résultat.
\begin{proof}[Démonstration de la Proposition \ref{po:global}]
 Par symétrie, l'étude de la convergence de la suite de mesures de Dirac dans la région $S$ revient au même problème dans la région $R$. Pour $u^\prime,e_i^\prime, a,b$ fixés, on définit l'ensemble  $F^\prime_{u^\prime,e_i^\prime,\varepsilon}(a,b)$ comme un sous-ensemble de $F^\prime(u^\prime,e_i^\prime,\varepsilon)$ (\ref{eq:coutingwithfixedparameters2}) de la même manière que pour (\ref{eq:conditionaa})(\ref{eq:conditionbb}).
 Le lemme précédent fournit que
$$\sum_{\substack{(a,b)\in H(e_i^\prime)\\
   \pgcd(a,b)=1\\ b\leqslant GB^{\frac{1}{6}}}}\\\sharp F^\prime_{u^\prime,e_i^\prime,\varepsilon}(a,b)\leqslant
   \sum_{\substack{(a,b)\in H(e_i)\\
   \pgcd(a,b)=1\\ b> DB^{\frac{1}{6}}}}\\\sharp F_{u,e_i,\varepsilon}(a,b)\leqslant
   \sum_{\substack{(a,b)\in H(e_i^\prime)\\
   \pgcd(a,b)=1\\ b\leqslant EB^{\frac{1}{6}}}}\\\sharp F^\prime_{u^\prime,e_i^\prime,\varepsilon}(a,b).$$
   D'après la Proposition \ref{po:Cremona2}, $(e_1^\prime,e_2^\prime,e_3^\prime)=(e_3,e_2,e_1)$, $u^\prime=u$ et donc
   $$e=e_1e_2e_3=e_1^\prime e_2^\prime e_3^\prime.$$
   D'après le Lemme \ref{le:constante}, 
   $$\sum_{\substack{(a,b)\in H(e_i^\prime)\\
   \pgcd(a,b)=1\\ b\leqslant EB^{\frac{1}{6}}}}\\\sharp F^\prime_{u^\prime,e_i^\prime,\varepsilon}(a,b)-
   \sum_{\substack{(a,b)\in H(e_i^\prime)\\
   \pgcd(a,b)=1\\ b\leqslant GB^{\frac{1}{6}}}}\\\sharp F^\prime_{u^\prime,e_i^\prime,\varepsilon}(a,b)=O_{u,e,\varepsilon}(B^{\frac{1}{3}}).$$
   Il découle de la Proposition \ref{po:globalcasi} que 
$$\sum_{\substack{(a,b)\in H(e_i^\prime)\\
   \pgcd(a,b)=1\\ b\leqslant GB^{\frac{1}{6}}}}\sharp F^\prime_{u^\prime,e_i^\prime,\varepsilon}(a,b)
   =\frac{Z(\varepsilon,u,e)}{\pi^2}B^{\frac{1}{3}}\log B + O_{u,e,\varepsilon}(B^{\frac{1}{3}}\log \log B).
$$
Ainsi
\begin{align*}
 &\sum_{\substack{(a,b)\in H(e_i)\\ 0<-a<b\\ \pgcd(a,b)=1}}\sharp F_{u,e_i,\varepsilon}(a,b)
 =\sum_{\substack{(a,b)\in H(e_i)\\
   \pgcd(a,b)=1\\ b\leqslant DB^{\frac{1}{6}},
    F_{u,e_i,\varepsilon}(a,b)\neq \varnothing  
}}+\sum_{\substack{(a,b)\in H(e_i) \\
   \pgcd(a,b)=1\\ b> DB^{\frac{1}{6}},
   F_{u,e_i,\varepsilon}(a,b)\neq \varnothing  
}}\sharp (F_{u,e_i,\varepsilon}(a,b))\\
&=2\sum_{\substack{(a,b)\in H(e_i)\\
   \pgcd(a,b)=1\\ b\leqslant DB^{\frac{1}{6}},
   F_{u,e_i,\varepsilon}(a,b)\neq \varnothing  
}}\sharp (F_{u,e_i,\varepsilon}(a,b))+O_{u,e,\varepsilon}(1)\\
 &=2\frac{Z(\varepsilon,u,e)}{\pi^2}B^{\frac{1}{3}}\log B + O_{u,e,\varepsilon}(B^{\frac{1}{3}}\log\log B).
\end{align*}

\end{proof}
Enfin pour obtenir la formule de $\rho_{B,Q}(\chi(\varepsilon))$, il ne reste qu'à sommer sur l'ensemble fini des $u,e_i$ possibles.

 \section{Détermination de la mesure asymptotique}
 
Pour $u,e_1,e_2,e_3$ fixé,
en vertu de ce que nous avons obtenu, asymptotiquement tous les points sont dans l'intérieur de la région 
\begin{align*}
T_{u^3e}&=\left\{(s,t)\in\mathbb{R}^2: s\left(\frac{1}{2}-\frac{1}{2}\sqrt{1-\frac{4eu^3}{s^3}}\right)<
t<s\left(\frac{1}{2}+\frac{1}{2}\sqrt{1-\frac{4eu^3}{s^3}}\right)\right\}\\
&=\{(s,t)\in\mathbb{R}^2:st(s-t)>u^3e\}.
\end{align*}
Pour déterminer la mesure asymptotique, it suffit de trouver son expression sur un ensemble de domaines simples qui permettent de tester la convergence de la famille de mesures.
Pour cela on commence par considérer les fonctions caractéristique des domaines de la forme suivante.
Pour $\varepsilon_2>\varepsilon_1>(4e)^{\frac{1}{3}}u$ et $\tau_2,\tau_1\in \mathopen[\frac{1}{2},\frac{1}{2}+C_0(u,e,\varepsilon_1)\mathclose[,\tau_2>\tau_1$ 
on considère la région
$$T_{\varepsilon_1,\varepsilon_2,\tau_1,\tau_2}=\{(s,t)\in T_{u^3e}: \tau_1<\frac{t}{s}<\tau_2,\varepsilon_1<s<\varepsilon_2\}.$$
Pour calculer le nombre total des points rationnels dans cette région,
il suffit de compter l'ensemble suivant.
Pour $\varepsilon>(4e)^{\frac{1}{3}}u$ et $\tau\in \mathopen ]\frac{1}{2},\frac{1}{2}+C_0\mathclose[$, on définit
 \begin{equation}
       F_{\tau,\varepsilon,u,e_i}=\left\{
 \begin{aligned}
  &P=[x:y:z]\\
  &x>y>z>0\\
  &\pgcd(x,y,z)=1
 \end{aligned}
\left|
\begin{aligned}
 &\left|\frac{x}{z}-1\right|\leqslant \varepsilon B^{-\frac{1}{3}}\\
 &H(P)\leqslant B,
 \mu(P)>\tau\\
 &u(P)=u,\e(P)=(e_1,e_2,e_3)
\end{aligned}
\right\}\right.
      \end{equation}
Alors la distribution locale de ce petit domaine peut être calculée par la différence de ceux des quatre domaines suivants: 
\begin{equation}\label{eq:penteetdistance}
 \varrho_{B,Q}(\chi_{T_{\varepsilon_1,\varepsilon_2,\tau_1,\tau_2}})=
\sum_{u,e_i}\left(\sharp F_{\tau_1,\varepsilon_2,u,e_i}-\sharp F_{\tau_2,\varepsilon_2,u,e_i}
-\sharp F_{\tau_1,\varepsilon_1,u,e_i}+\sharp F_{\tau_2,\varepsilon_1,u,e_i}\right).
\end{equation}
    En utilisant notre paramétrisation des droites, on a comme précédemment
      $$\sharp F_{\tau,\varepsilon,u,e}
=\sum_{\substack{0<-a<b,-\frac{a}{b}>\tau\\\pgcd(a,b)=1\\e_1|a,e_2|b,e_3|a+b}}F_{u,e_i,\varepsilon}(a,b).$$
 \begin{propo}\label{po:pente1}
 On a pour $u,e_1,e_2,e_3$ fixé,
  \begin{equation}
   \sum_{\substack{0<-a<b,-\frac{a}{b}>\tau\\\pgcd(a,b)=1\\e_1|a,e_2|b,e_3|a+b}}F_{u,e_i,\varepsilon}(a,b)
   =\frac{2}{\pi^2}Z(\tau,u,e,\varepsilon)B^{\frac{1}{3}}\log B+O_{\tau,u,e,\varepsilon}(B^{\frac{1}{3}}\log \log B).
  \end{equation}
  où
  $$Z(\tau,u,e,\varepsilon)=\sum_{k|u}\frac{\phi(ue)\mu(k)3^{\omega(k)}}{ k}\left(\frac{1}{e^{\frac{1}{3}}}\int_{\theta\in \mathopen ]\tau,\frac{1}{2}+C_0\mathclose [} \frac{\operatorname{d}\theta}{(\theta(1-\theta))^{\frac{2}{3}}}-\frac{u}{\varepsilon}\int_{\theta\in \mathopen ]\tau,\frac{1}{2}+C_0\mathclose [} \frac{\operatorname{d}\theta}{(\theta(1-\theta))}\right).$$
\end{propo}
Toutes les méthodes que nous avons développées peuvent être appliquées. En outre, il faut étudier le comportement de la pente sous la transformation de Cremona.
\begin{lemme}\label{le:Cremona3}
 Conservons les notations $E=E(u,e,\varepsilon)$ et $G=G(u,e,\varepsilon)$ du Lemme \ref{le:constante}, et $F(u,e_i,\varepsilon)$, $F^\prime(u,e_i,\varepsilon)$ définis par (\ref{eq:coutingwithfixedparameters}) et (\ref{eq:coutingwithfixedparameters2}) respectivement. Soient
 $$B_1=\{V\in F(u,e_i,\varepsilon):n(V)>DB^{\frac{1}{6}},\mu(V)>\tau\},$$
$$B_2^\prime=\{U\in  F^\prime(u,e_i,\varepsilon): n^\prime(U)\leqslant EB^{\frac{1}{6}},\mu^\prime(U)<1-\tau\}, $$
$$B_3^\prime=\{U\in  F^\prime(u,e_i,\varepsilon): n^\prime(U)\leqslant GB^{\frac{1}{6}},\mu^\prime(U)<1-\tau-2\varepsilon B^{-\frac{1}{3}}\}.$$
Alors on a pour $B\gg_{u,e,\varepsilon} 1$,
 \begin{equation}
   B_3^\prime\subseteq \Psi(B_1)\subseteq B_2^\prime.
 \end{equation}
\end{lemme}
\begin{proof}[Démonstration]
	Il suffit de le vérifier pour la pente grâce aux Proposition \ref{po:Cremona1} et \ref{po:Cremona2}.  
	Prenons $V=[x:y:z]\in B_1$ et notons $\Psi(V)=U\in S$. Puisque $x>y>z$, on a
	$$\mu^\prime(U)=\frac{xz-yz}{xy-yz}=\frac{z}{y}\frac{x-y}{x-z}=\frac{z}{y}(1-\mu(V))<1-\tau.$$
	Cela démontre que $\Psi(B_1)\subseteq B_2^\prime$.
	Maintenant prenons $U=[x^\prime :y^\prime:z^\prime]\in B_3^\prime$ et notons $\Psi(U)=V\in R$.
	Puisque $$\frac{y^\prime}{x^\prime}<\frac{z^\prime}{x^\prime}\leqslant 1+\varepsilon B^{-\frac{1}{3}},$$ on a 
	$$\frac{x^\prime}{y^\prime}>\frac{1}{1+\varepsilon B^{-\frac{1}{3}}}>1-\varepsilon B^{-\frac{1}{3}}.$$ Donc
	$$\mu(V)=\frac{x^\prime z^\prime -x^\prime y^\prime}{y^\prime z^\prime-x^\prime y^\prime}=\frac{x^\prime}{y^\prime}\frac{z^\prime-y^\prime}{z^\prime-x^\prime}=\frac{x^\prime}{y^\prime}(1-\mu^\prime(U))>(\tau+2\varepsilon B^{-\frac{1}{3}})(1-\varepsilon B^{-\frac{1}{3}})>\tau.$$
	Cela démontre que $\Psi(B_3^\prime)\subseteq B_1$ et d'où $B_3^\prime\subseteq \Psi(B_1)$.
\end{proof}

On a aussi à répéter le Lemme \ref{le:echange}.
\begin{lemme}\label{le:echange2}
Sous l'hypothèse du Lemme \ref{le:echange}, on suppose $B\gg_{u,e,\varepsilon} 1$ de sorte que $\frac{1}{2}+C_1>\tau$ et on définit
 $$S^\prime(\tau,M)=S^\prime(\tau,e,u,\varepsilon,M)=\{(m,n)\in \mathbb{Z}^2: 1\leqslant m<n\leqslant M,\frac{m}{n}\in \mathopen ]\tau,\frac{1}{2}+C_1\mathclose [\},$$
  $$T(\tau,M)=T(\tau,e,u,\varepsilon,M)=\{(x,y)\in\mathbb{R}^2:1\leqslant y< x\leqslant M, \frac{y}{x}\in \mathopen ]\tau,\frac{1}{2}+C_0\mathclose [\}.$$
 Alors on a
 $$\sum_{(m,n)\in S^\prime(\tau,M)\cap H(\alpha_i)}
 (|m|n(m+n))^{-\frac{2}{3}}=\iint_{T(\tau,M)}
 \frac{\operatorname{d}x\operatorname{d}y}{\alpha (xy(x-y))^{\frac{2}{3}}}+O_{u,e,\alpha,\varepsilon}(\log \log B),$$
 $$\sum_{(m,n)\in S^\prime(\tau,M)\cap H(\alpha_i)}
 (m(n-m))^{-1}=\iint_{T(\tau,M)}
 \frac{\operatorname{d}x\operatorname{d}y}{\alpha y(x-y)}+O_{u,e,\alpha,\varepsilon}(\log \log B).$$
\end{lemme}

\begin{proof}[Démonstration de la Proposition \ref{po:pente1}]
 On écrit \og$O$\fg\ pour \og$O_{u,e,\varepsilon,\tau}$\fg. Comme précédemment on considère les deux cas: $b\leqslant DB^{\frac{1}{6}}$ et $b>DB^{\frac{1}{6}}$.
 Pour le premier cas, la démonstration de la Proposition \ref{po:globalcasi} nous fournit le résultat
$$\sum_{\substack{0<-a<b,-\frac{a}{b}>\tau\\\pgcd(a,b)=1,b\leqslant DB^{\frac{1}{6}}\\e_1|a,e_2|b,e_3|a+b}}\sharp 
 F_{u,e_i,\varepsilon}(a,b)=\pi^{-2}Z(\tau,u,e,\varepsilon) B^{\frac{1}{3}}\log B + O_{u,e,\varepsilon,\tau}(B^{\frac{1}{3}}\log \log B).
$$
Pour l'autre cas, grâce à la transformation de Cremona, on obtient d'après le Lemme \ref{le:Cremona3} que
$$\sum_{\substack{0<-a<b,-\frac{a}{b}>\tau\\\pgcd(a,b)=1,b>DB^{\frac{1}{6}}\\e_1|a,e_2|b,e_3|a+b}}\sharp F^\prime_{u,e_i,\varepsilon}(a,b)\leqslant \sum_{\substack{0<-a<b,-\frac{a}{b}<1-\tau\\\pgcd(a,b)=1,b\leqslant EB^{\frac{1}{6}}\\e_3|a,e_2|b,e_1|a+b}}\sharp F^\prime _{u,e_i,\varepsilon}(a,b),$$
$$\sum_{\substack{0<-a<b,-\frac{a}{b}>\tau\\\pgcd(a,b)=1,b>DB^{\frac{1}{6}}\\e_1|a,e_2|b,e_3|a+b}}\sharp F_{u,e_i,\varepsilon}(a,b)\geqslant  \sum_{\substack{0<-a<b,-\frac{a}{b}<1-\tau-2\varepsilon B^{-\frac{1}{3}}\\\pgcd(a,b)=1,b\leqslant EB^{\frac{1}{6}}\\e_3|a,e_2|b,e_1|a+b}}\sharp F^\prime_{u,e_i,\varepsilon}(a,b).$$
Du point de vue du Lemme \ref{le:echange2}, le même raisonnement que la Proposition \ref{po:globalcasi} nous donne
\begin{align*}
 &=\sum_{\substack{0<-a<b,-\frac{a}{b}<1-\tau\\\pgcd(a,b)=1,b\leqslant EB^{\frac{1}{6}}\\e_3|a,e_2|b,e_1|a+b}}\sharp F^\prime_{u,e_i,\varepsilon}(a,b)\\
 &=\sum_{k|u}\frac{\phi(ue)\mu(k)3^{\omega(k)}}{\pi^2 k}\left(\frac{1}{e^{\frac{1}{3}}}\int_{\theta\in \mathopen ]\frac{1}{2}-C_0,1-\tau\mathclose [} \frac{\operatorname{d}\theta}{(\theta(1-\theta))^{\frac{2}{3}}}-\frac{u}{\varepsilon}\int_{\theta\in \mathopen ]\frac{1}{2}-C_0,1-\tau [} \frac{\operatorname{d}\theta}{(\theta(1-\theta))}\right)B^{\frac{1}{3}}\log B\\
 &\ +O(B^{\frac{1}{3}}\log \log B)\\
&=\sum_{k|u}\frac{\phi(ue)\mu(k)3^{\omega(k)}}{ \pi^2 k}\left(\frac{1}{e^{\frac{1}{3}}}\int_{\theta\in \mathopen ]\tau,\frac{1}{2}+C_0\mathclose [} \frac{\operatorname{d}\theta}{(\theta(1-\theta))^{\frac{2}{3}}}-\frac{u}{\varepsilon}\int_{\theta\in \mathopen ]\tau,\frac{1}{2}+C_0 [} \frac{\operatorname{d}\theta}{(\theta(1-\theta))}\right)B^{\frac{1}{3}}\log B\\
&\ +O(B^{\frac{1}{3}}\log \log B)\\
&=\pi^{-2}Z(\tau,u,e,\varepsilon) B^{\frac{1}{3}}\log B + O_{u,e,\varepsilon,\tau}(B^{\frac{1}{3}}\log \log B).
\end{align*}
De l'autre côté, grâce à la symétrie de la transformation de Cremona (Lemme \ref{le:Cremona3}), le même procédé nous fournit
\begin{align*}
&=\sum_{\substack{0<-a<b,-\frac{a}{b}<1-\tau-2\varepsilon B^{-\frac{1}{3}}\\\pgcd(a,b)=1,b\leqslant EB^{\frac{1}{6}}\\e_3|a,e_2|b,e_1|a+b}}\sharp F_{u,e_i,\varepsilon}(a,b)\\
&=\sum_{k|u}\frac{\phi(ue)\mu(k)3^{\omega(k)}}{\pi^2 k}\left(\frac{1}{e^{\frac{1}{3}}}\int_{\theta\in \mathopen ]\frac{1}{2}-C_0,1-\tau-2\varepsilon B^{-\frac{1}{3}}\mathclose [} \frac{\operatorname{d}\theta}{(\theta(1-\theta))^{\frac{2}{3}}}-\frac{u}{\varepsilon}\int_{\theta\in \mathopen ]\frac{1}{2}-C_0,1-\tau-2\varepsilon B^{-\frac{1}{3}} [} \frac{\operatorname{d}\theta}{(\theta(1-\theta))}\right)B^{\frac{1}{3}}\log B\\
&\ +O(B^{\frac{1}{3}}\log \log B)\\
&=\sum_{k|u}\frac{\phi(ue)\mu(k)3^{\omega(k)}}{ \pi^2 k}\left(\frac{1}{e^{\frac{1}{3}}}\int_{\theta\in \mathopen ]\tau,\frac{1}{2}+C_0\mathclose [} \frac{\operatorname{d}\theta}{(\theta(1-\theta))^{\frac{2}{3}}}-\frac{u}{\varepsilon}\int_{\theta\in \mathopen ]\tau,\frac{1}{2}+C_0 [} \frac{\operatorname{d}\theta}{(\theta(1-\theta))}\right)B^{\frac{1}{3}}\log B\\
&\ +O(\log B)+O(B^{\frac{1}{3}}\log \log B)\\
&=\pi^{-2}Z(\tau,u,e,\varepsilon) B^{\frac{1}{3}}\log B + O_{u,e,\varepsilon,\tau}(B^{\frac{1}{3}}\log \log B).
\end{align*}

Cela termine la preuve.
\end{proof}

\begin{proof}[Démonstration du Théorème \ref{th:2}]
 Lorsque $\varepsilon_2-\varepsilon_1$ et $\tau_2-\tau_1$ sont suffisamment petits, d'après la Proposition précédente et (\ref{eq:penteetdistance}) on a
 \begin{align*}
  \frac{\varrho_{B,Q}(\chi_{T_{\varepsilon_1,\varepsilon_2,\tau_1,\tau_2}})}{B^{\frac{1}{3}}\log  B}
  &=\sum_{u,e}\frac{2}{\pi^2}\left(Z(\tau_1,u,e,\varepsilon_2)-Z(\tau_2,u,e,\varepsilon_2)-Z(\tau_1,u,e,\varepsilon_1)+Z(\tau_2,u,e,\varepsilon_1)\right)\\
  &=\sum_{\substack{u,e_i\\\tau_1<\frac{1}{2}+C_0(u,e,\varepsilon_1)}}\frac{2u\phi(ue)}{\pi^2}\sum_{k|u}\frac{\mu(k)3^{\omega(k)}}{ k}\left(\varepsilon_1^{-1}-\varepsilon_2^{-1}\right)\int_{\tau_1}^{\tau_2} \frac{\operatorname{d}\theta}{\theta(1-\theta)}+o(1)\\
  &=\sum_{\substack{u,e\\yx^{-1}<\frac{1}{2}+C_0(u,e,x)}}\frac{2u\phi(ue)3^{\omega(e)}}{\pi^2}\sum_{k|u}\frac{\mu(k)3^{\omega(k)}}{ k}\iint_{\substack{s\in [\varepsilon_1,\varepsilon_2]\\ \frac{t}{s}\in ]\tau_1,\tau_2[}}\frac{\operatorname{d}s \operatorname{d} t}{st(s-t)}+o(1)\\
 &= \sum_{\substack{u,e\\u^3e<st(s-t)}}\frac{2u\phi(ue)3^{\omega(e)}}{\pi^2}\sum_{k|u}\frac{\mu(k)3^{\omega(k)}3^{\omega(e)}}{  k}\iint_{T_{\varepsilon_1,\varepsilon_2,\tau_1,\tau_2}}\frac{\operatorname{d}s \operatorname{d} t}{st(s-t)}+o(1).
 \end{align*}
 Toute fonction continue de support compact est la limite uniforme d'une suite de fonctions intégrables qui sont des combinaisons linéaires des fonctions caractéristiques des ensembles définies au début.
 Cela termine la preuve. 
\end{proof}
\section{\`{A} la recherche des interprétations}
Le résultat révèle qu'il y a des \og seuils\fg $$\{(s,t)\in\mathbb{R}^2_{>0}:st(s-t)=u^3e\}$$ 
correspondant à des discontinuités de la mesure sur le bord de la région $T_{u^3 e}$.
Le nombre de possibilités pour $u,e$ croît avec $x,y$.
On explique maintenant en quoi cela provient d'un phénomène similaire pour $\mathbb{P}^1$.
Rappelons ce que l'on a fait, on compte sur chaque droite et on somme leurs contributions.
Cela équivaut à considérer ce problème sur chaque droite avec la hauteur et la distance induite.
On rappelle le résultat sur la distribution asymptotique pour $\mathbb{P}^1$ (\cite{Pagelot}).
On fixe la distance et la hauteur sur $l$ définies par $$d_l([u:v],[0:1])=\left|\frac{u}{v}\right|,
\quad H_l([u:v])=H_{\mathcal{O}(1)}([u:v])=\max(|u|,|v|), \quad (\pgcd(u,v)=1).$$ 
Il est facile de voir que la constante d'approximation essentielle est $1$ et aucune sous-variété
n'est localement accumulatrice.
On peut donc considérer le même problème d'étudier la distribution locale autour du point $O=[0:1]$.
En identifiant l'espace tangent $T_O \mathbb{P}^1$ de $O$ avec $\mathbb{A}^1$ canoniquement, on veut calculer la mesure asymptotique
de la suite $\{\lambda_B\}$ où 	Pour toute fonction $f$ intégrable à support compact,
$$\lambda_B(f)=\sum_{[u:v]\in F(\varepsilon, B)}f([u:v])$$
avec
$$F(\varepsilon,B)=\{P=[u:v]:d(P,O)<\varepsilon B^{-1},H(P)\leqslant B\}.$$
\begin{theo}[Pagelot]
 On a 
 $$\lambda_B(f)=B\int f d\lambda +o(1),$$
 où la mesure asymptotique est donnée par
 $$\lambda =\frac{\sigma(x)}{x^2}\operatorname{d}x, \quad \sigma(t)=\sum_{n\leqslant t} \varphi(n).$$
\end{theo}
On voit que les seuils pour la mesure asymptotique sont $\{x=n\} ~(n\in\mathbb{N}^*)$.

Revenons à notre cas. On fixe une droite $l:ax+by+c=0$. Le seuil restreint à $l$ est
$$\{(s,t)\in l\cap\mathbb{R}^2_{>0} :st(s-t)=u^3e\}=\{(s,t)\in l\cap\mathbb{R}^2_{>0}:s^3\mu_l(1-\mu_l)=u^3e\}$$
où $\mu_l$ est la pente de $l$ égale à $-ab^{-1}$.
Rappelons que notre paramétrisation est donnée par
$$f_l:[u:v]\longmapsto [ub+v:-ua+v:v].$$
On a $f([0:1])=Q$.
Comme asymptotiquement 
$$H_{\omega_X^{-1}}|_l\sim \frac{e_1e_2e_3}{|ab(a+b)|}H^3_l, $$
$$d|_l=\max(|a|,|b|) d_l,$$
les constantes devant $d_l$ et $H_l$ modifient les seuils si l'on définit la suite de mesures par rapport à $H_{\omega_X^{-1}}|_l$ et $d|_l$. 
D'après (\ref{eq:classique}), on en déduit que la mesure asymptotique dans ce cadre est donnée par
$$\lambda^\prime_l=\frac{1}{b^2(\mu_l(1-\mu_l))}\frac{\sigma_l^\prime(x)}{x^2}\operatorname{d}x,$$
où
$$\sigma^\prime_l(t)=\sum_{u,e:u^3e\leqslant \mu_l(1-\mu_l)t^3}\varphi(ue).$$
Les seuils de $\sigma_l^\prime$  correspondent exactement au bord de $T_{u^3 e}$ restreint à $l$. Quand $B\to\infty$, le seuil \textit{discret} qui est formé par les droites qui interviennent tend vers le seuil \textit{continu} de la mesure asymptotique de la surface $X_3$.

\section{Résultats pour d'autres surfaces toriques}
Notre méthode peut être utilisée pour traiter les surfaces $\mathbb{P}_1, X_1,X_2$ (rappelons que $X_1,X_2$ désignent les surfaces de del Pezzo en éclatant $1$ ou $2$ points généraux dans $\mathbb{P}^2$).
La différence principale entre elles et $X_3$ est la suivante.
Pour $X_3$, pour un voisinage $U$ de $Q$ (de $T_Q X_3$) et $B$ fixés, le nombre des droites $ax+by+cz=0$ qui intervienne dans le comptage
croît linéairement par rapport à $B^{\frac{1}{3}}$. Cela explique que la mesure asymptotique doit avoir une densité.
Alors que pour les surfaces $\mathbb{P}_1, X_1,X_2$, ce nombre est uniformément borné 
\textit{a posteriori}(dépend seulement de $U$ mais indépendant de $B$).
Nous esquissons la preuve pour $X_2$.

Soit $X_2$ l'éclatement de $\mathbb{P}^2$ en $[1:0:0]$ et $[0:1:0]$.
Par un argument similaire à celui de la section 2, la constante d'approximation essentielle pour $Q=[1:1:1]$ est $3$.
Le fibré anticanonique très ample donne un plongement de $X_2$ dans $\mathbb{P}^7$.
Une hauteur de Weil absolue du point $P=[x:y:z]$ en dehors des diviseurs exceptionnels est donnée par
$$H(P)=H_{\omega_{X_2}^{-1}}(P)=\max(|x|,|y|,|z|)\frac{\max(|y|,|z|)\max(|x|,|z|)}{\pgcd(y,z)\pgcd(x,z)}.$$
Les variétés localement accumulatrice que l'on doit supprimer sont les droites $$l_1=\{y=z\},\quad l_2=\{x=z\}.$$
On identifie l'espace tangent $T_Q X_2$ avec $(z\neq 0)\simeq\mathbb{A}^2$ et on étudie la distribution locale dans la région $$\bar{R}=\{(s,t)\in\mathbb{R}:s,t>1\}.$$
En introduisant la même paramétrisation pour les droites, on se ramène au problème de comptage
$$\varrho_{B,Q}(\chi_\varepsilon)=\sum_{\substack{0<-a,b\\\pgcd(a,b)=1}} \sharp J_B(a,b)$$
avec (pour $0<-a\leqslant b$)
\begin{align}
 J_B(a,b) =\left\{  
 \begin{aligned}
 & (u,v)\in \mathbb{N}^2,\\
 &\pgcd(u,v)=1
 \end{aligned}
 \left| 
 \begin{aligned}
  &\frac{u}{v}\leqslant\varepsilon B^{-\frac{1}{3}}b^{-1}\\
  &(ub+v)^2(-ua+v)\leqslant Bd_1d_2
 \end{aligned}
\right\} \right.
 \end{align}
 On note $$e_1=\frac{a}{d_1},\quad e_2=\frac{b}{d_2},\quad u=d_4=\pgcd(x-z,y-z).$$
 On a pour $0<-a\leqslant b$,
 $$b\leqslant \varepsilon B^{-\frac{1}{3}}vu^{-1}\leqslant \varepsilon u^{-1}(d_1d_2)^{\frac{1}{3}}\leqslant \varepsilon b^{\frac{2}{3}}u^{-1}(e_1e_2)^{-\frac{1}{3}}.$$
 Alors
 $$bu^3e_1e_2\leqslant \varepsilon ^3.$$
 On voit que non seulement $e_1,e_2,u$ sont bornées, mais $b$ (et \textit{a fortiori} $a$) le sont aussi.
 
Étant donnés $u,a,b,e_i$, à une inversion de Möbius près, l'ensemble de $v$ dans $J_B$ est l'intersection avec un réseau de l'intervalle
 $$\mathopen [u\varepsilon^{-1} B^{\frac{1}{3}},(Bd_1d_2)^{\frac{1}{3}}-ub C^\prime _{a,b,e_i,B,\varepsilon}\mathclose ].$$
 Puisque l'ensemble des paires $(a,b)$ telles que cette intervalle soit non-vide est \textit{fini} uniformément, 
 le terme principal va avoir la grandeur $B^{\frac{1}{3}}$ et 
 le terme d'erreur peut être contrôlé facilement.
 Donc on retrouve le résultat que Pagelot avait énoncé dans \cite{Pagelot}.
 \begin{theo}
  Soit $Y$ une surface de del Pezzo de degré $\geqslant 7$. Alors pour toute fonction $f$ intégrable à support compact, on a 
  $$\int f\operatorname{d}\varrho_{B,Q}=B^{\frac{1}{3}}\int f \operatorname{d}\varrho_{Y,Q}+o(1),$$ où la mesure asymptotique est donnée par
  $$\varrho_{Y,Q}=\sum_{\substack{l\ni Q\\(\omega_X^{-1}\cdot l)=3}} \lambda_l,$$
  où $\lambda_l$ est une mesure asymptotique sur la droite $l$ (ayant des seuils et des valeurs différents selon le degré de $Y$, voir section 6.1).
  De plus, les termes $\lambda_l(f)$ sont presque tous nuls.
 \end{theo}
La grandeur de ces trois exemples est $B^{\frac{1}{3}}$, qui est inférieure à celle de $X_3$, ce qui est sous-entendu parce que plus on éclate des points, plus la hauteur diminue, ainsi plus les points s'accumulent. De plus, la mesure asymptotique, contrairement à celle de $X_3$, n'a pas une densité continue (par rapport à la mesure de Lebesgue de $\mathbb{R}^2$ dans $T_Q Y$).

\section*{Remerciement.} Je tiens à remercier Emmanuel Peyre pour ses conseils éclairants et corrections attentives, qui ont beaucoup amélioré la clarté de ce texte.

 \end{document}